\documentclass{birkjour}
\usepackage{pgf}
\usepackage{tikz}
\usetikzlibrary{arrows,automata}
\usepackage{circuitikz}
\newtheorem{thm}{Theorem}[section]
\newcommand{\be}{\begin{eqnarray}}
\newcommand{\ee}{\end{eqnarray}}
\newcommand{\bd}{\begin{displaymath}}
\newcommand{\ed}{\end{displaymath}}
\newcommand {\bad}{\begin{aligned}}
\newcommand {\ead}{\end{aligned}}

\newtheorem{n.b}[thm]{Note}
\newtheorem{lem}[thm]{Lemma}
\newtheorem{exa}[thm]{Example}

\newcommand{\bal}{\mbox{\boldmath $\alpha$}}
\newcommand{\bbet}{\mbox{\boldmath $\beta$}}

\newcommand{\bphi}{\mbox{\boldmath $\phi$}}
\newcommand{\bPhi}{\mbox{\boldmath $\Phi$}}
\newcommand{\bPsi}{\mbox{\boldmath $\Psi$}}
\newcommand {\bdm}{\begin{displaymath}}
\newcommand {\edm}{\end{displaymath}}
\newcommand {\ben}{\begin{equation}}
\newcommand {\een}{\end{equation}}

\newtheorem{example}{Example}[section]
\newtheorem{theorem}{Theorem}[section]
\newtheorem{remark}{Remark}[section]

\newtheorem{lemma}{Lemma}[section]

\newfont{\smoldita}{cmmib8}
\newfont{\boldita}{cmmib10}
\newfont{\bboldita}{cmmib10 scaled\magstep1}
\newcommand{\sem}[1]{\mbox{$\{e^{t#1}\}_{t \geq 0}$}}

\newcommand{\mbb}[1]{\mathbb{#1}}
\newcommand{\mb}[1]{\mathbf{#1}}
\newcommand{\init}[1]{\stackrel{\circ}{#1}}
\newcommand{\nn}{\nonumber}
\newcommand{\e}{\epsilon}

\newcommand{\p}{\partial}

\newcommand{\cl}[2]{\int\limits_{#1}^{#2}}

\newcommand{\bom}{\mbox{\boldmath $\omega$}}

\newcommand{\ti}[1]{\tilde{#1}}

\newcommand{\la}{\lambda}

\newcommand{\mc}[1]{\mathcal{#1}}

\begin{document}

\markboth{J. Banasiak, A. Falkiewicz \& P. Namayanja}{Asymptotic state lumping in  transport and diffusion problems on networks}

\title{Asymptotic state lumping in transport and diffusion problems on networks}

\author{J. BANASIAK}

\address{School of Mathematics, Statistics and Computer Science, University of
KwaZulu-Natal, Durban, South Africa
 \\
 Institute of Mathematics,
Technical University of \L\'{o}d\'{z}, \L\'{o}d\'{z}, Poland\\
banasiak@ukzn.ac.za}
\thanks{Research of J.B. and A.F was done during NRF/IIASA SA YSSP at the University of Free State and was partly supported by  National Science Centre of Poland through the grant N N201605640. Research of P.N. was supported by TWOWS and the UKZN Research Fund.}
\author{A. FALKIEWICZ}

\address{Institute of Mathematics,
\L\'{o}d\'{z} University of Technology, \L\'{o}d\'{z}, Poland\\
aleksandrafalkiewicz@gmail.com}

\author{P. NAMAYANJA}
\address{School of Mathematics, Statistics and Computer Science, University of
KwaZulu-Natal, Durban, South Africa
 \\
proscovia@aims.ac.za}



\maketitle

\begin{abstract}
One of the aims of systems biology is to build  multiple layered and multiple scale models of living systems which can efficiently describe phenomena occurring at various level of resolution. Such models should consist of layers of various microsystems interconnected by a network of pathways, to  form a macrosystem in a consistent way; that is, the observable characteristics of the macrosystem should be, at least asymptotically,  derivable by aggregation of the appropriate features of the microsystems forming it, and from the properties of the network. In this paper we consider a general macromodel describing a population consisting of several interacting with each other subgroups, with the rules of interactions  given by a system of ordinary differential equations, and we construct two different micromodels whose aggregated dynamics is approximately the same  as that of the original macromodel. The micromodels offer a more detailed description of the original macromodel's dynamics by considering an internal structure of each subgroup. Here, each subgroup is represented by an edge of a graph with diffusion or transport occurring along it, while the interactions between the edges are described by  interface conditions at the nodes joining them. We prove that with an appropriate scaling of such models, roughly speaking, with fast diffusion or transport combined with slow exchange at the nodes, the solutions of the micromodels are close to the solution to the macromodel.
\end{abstract}

\section{Introduction}	
\label{sec1}
In the recent paper \cite{bel}, the authors have proposed an interesting and exciting paradigm of developing theoretical biology through a proper mathematization of systems biology. It involves an interplay of many disciplines belonging to mathematics, biology and to their overlap, such as the theory of evolution, immune competition, mutation and selection, kinetic theory, evolutionary game theory, multi-scale methods and networks.
In the words of D. Noble, Ref.~ \cite{No}, pp. 112 and 129,
\begin{quote}
We are looking towards a mature theory of biological systems-level interactions [...]. The task of systems biology is first to unravel these interactions and then to develop theories to account for them, and so to lay bare their logical underpinnings.

One of the important goals of integrative systems biology is to identify the levels at which the various
functions exist and operate.
\end{quote}
In other words, the aim of systems biology is to build  universal, hierarchical models of  biological phenomena which would include
all levels of organization of living matter.  For instance, for malaria, one would like to have a model including the dynamics of the plasmodium, through the cell, tissue, individual and ending at the population, or even metapopulation, levels.

\subsection{Multi-scale models in systems biology.} Systems biology has reached out to many areas of mathematics, even including disciplines such as category theory \cite{Lo}, or logic and set theory \cite{Ra} which, for many years, have been regarded as belonging to pure mathematics. However, modelling solely based on these disciplines  results in a static picture of the system,  which is not satisfactory for many applications. To provide the required dynamical features we enhance the model by incorporating differential equations into the description. These equations are intended to model the  evolution of particular building blocks of the system  and typically they are interlinked by a complex network spanning several time and size scales. Usually the existence of different scales, or levels, in the model is revealed by the presence of small (or large) scaling parameters representing the ratios of the times typical for processes occurring in particular building blocks, or the ratios of the sizes of objects involved in them. In full generality, such models are too complex to allow for any robust analysis and thus it is of interest to be able to focus on the scale relevant to a particular aspect of the system's behaviour by selecting an appropriate simpler sub-model. However, in many applied sciences often there is an expectation that complex multi-scale systems can be described by plug-and-play type models. In such models, focusing on a required level amounts just to switching off the unwanted scales by setting relevant scaling parameters to 0. Unfortunately, in most cases it is impossible. Even in the text quoted above we see that there is a clear understanding that all levels of organization are interconnected. In other words, it is recognized that there is an interdependence between various scales in the model so that, whenever we focus on a sub-model acting at some level, there should be a `shadow' of all scales that were switched off.  The problem is exacerbated by the fact that dynamics at  different levels usually is described by equations which are not fully compatible with each other, especially close to the boundaries between the levels. In other words, setting a particular scaling parameter equal to zero may dramatically change the type of the equation and render the problem ill-posed. Thus, moving between the scales, if possible at all, cannot be accomplished by simply setting the appropriate scaling parameter to 0 but  requires complicated limit procedures which lead to the so-called singular limits of the involved equations.

\subsection{Multi-scale models on networks.} The paper Ref.~ \cite{bel} offers a survey of a wide range of mathematical methods which allow for dealing with such complex and interrelated models. In the presented paper we shall focus on particular aspects of  the proposed approach, namely on multi-scale dynamical systems related by a network and their singular limits. Such systems  provide a detailed description of a model at what we refer to as the micro-scale, while the singular limits give an aggregated description of the model at the macro-scale. This procedure is often referred to as the asymptotic state lumping. In this paper we focus on two types of dynamical systems, transport and diffusion and, to keep technicalities to the minimum, we restrict ourselves to linear problems.

To set the stage, we begin with a brief description of how complex, more detailed, systems are derived from more crude building blocks.   When we model a complex, interlinked system of subpopulations, we often begin by considering a static graph (or often a directed graph), where the edges represent connections between the subpopulations and the intensity of interactions are given by the weights associated with the edges, see e.g. Ref.~ \cite{deo}.
\begin{figure}
\centering
\begin{tikzpicture}[scale=0.5, ->,>=stealth',shorten >=1pt,auto,node distance=1.5cm,
                    semithick]
  \tikzstyle{every state}=[fill=none,draw=none,text=black]

  \node [state] (A)                  				  {$V_{1}$};
  \node[state]         (B) [right of=A]       {$V_{2}$};
  \node[state]         (C) [right of=B]       {$V_{3}$};
  \node[state]         (D) [below of=C]       {$V_{4}$};
	\node[state]         (E) [right of=C]       {$V_{5}$};
	\path     (A) edge      [loop above]        node [above] {$E_{1}$} (C);
\path     (B) edge      [loop above]        node [above] {$E_{3}$} (B);
\path     (C) edge      [loop above]        node [above] {$E_{5}$} (B);
\path     (D) edge      [loop left]        node [below] {$E_{9}$} (B);
\path     (E) edge      [loop above]        node [above] {$E_{8}$} (B);
  \path               (A) edge              node {$E_{2}$} (B);
  \path               (B) edge              node {$E_{4}$} (C);
	\path               (D) edge              node {$E_{10}$} (C);
  \path               (C) edge [bend right] node [below]{$E_{6}$} (E);
  \path               (E) edge [bend right]  node[above]{$E_{7}$} (C);
\end{tikzpicture}
\caption{An example of a directed graph.}\label{fig1}
\end{figure}
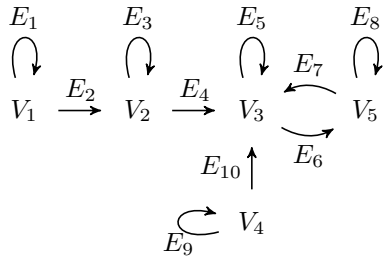
As we mentioned above, static models usually fall short of what is expected from them. Therefore the next step is to assume that a subpopulation, localized at a particular node,  changes due to interactions with subpopulations at the nodes connected with that node. It results in a system of ordinary differential equations, where, in the linear case, the weights are the rates at which a particular node influences the changes at the connected nodes.
\subsubsection{First examples}
\begin{example}
\textbf{A basic mutation model.}\label{exRot}
   Consider the population described by $\mb v =(v_1,\ldots, v_m),$ where $v_j, 1\leq j\leq m,$ is the number of cells whose genotype belongs to class $j$ (for instance, having $j$ copies of a specific gene). Then  its evolution can be described by the system
\begin{equation}
\p_t\mb v = \mbb K\mb v,
\label{ku}
\end{equation}
where $\mbb K$ is the matrix describing connections between the nodes. For the dynamics on the network on Fig. \ref{fig1}, $\mbb K$ can be given by    \begin{equation}
\mbb K = \begin{pmatrix} 1+ k_{11} &  0& 0&0&0 \\
k_{21} & 1+k_{22}  & 0&0&0 \\
0 & k_{32} & 1+ k_{33}&k_{34}& k_{35}\\
0 & 0 &0 &1+k_{44}&0 \\
0 & 0 & k_{53}&0& 1+k_{55}
\end{pmatrix},
\label{Rmodel}
\end{equation}
and describe a situation in which the cell in class $j$ divides into two daughter cells, one of which has the same genotype as the mother, while the other changes its genotype to that of class $i$ with probability $k_{ij}$. In such a case $(k_{ij})_{1\leq i,j\leq 5}$ is a column stochastic matrix. We note that a particular case of this model is the discrete Rotenberg-Rubinov-Lebowitz model \cite{rot} where the cells are divided in classes according to their maturation velocity.
\end{example}

\begin{example}\label{AG}\textbf{ Aristizabal and Glavinovi\v{c} model of synaptic depression.} In  \cite{AG} the authors introduced a heuristic model of synaptic depression. In this model, neurotransmitters are localized in three compartments, or pools: the large
pool, where also their synthesis takes place, the small, intermediate, pool, and
the immediately available pool, from which they are released during stimulus. The key assumption in \cite{AG} was that the dynamics of the densities $\mb U = (U_1,U_2,U_3)$ of vesicles
with neurotransmitters in the pools is analogous to that of voltages across the capacitors in the electric circuit.
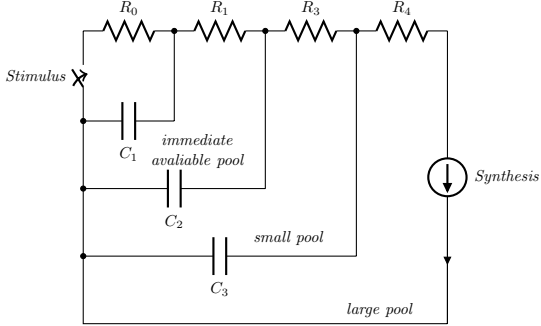
\begin{figure}
\begin{center}
\begin{circuitikz} [scale =0.6, transform shape] \draw
(0,2) to [R=$R_{0}$,-*] (2,2) -- (2,0)
			to [C=$C_{1}$,-*] (0,0)
			to [closing switch=Stimulus] (0,2)
(2,2) to [R=$R_{1}$,-*] (4,2) -- (4,-1.5)
			to [C=$C_{2}$,-*] (0,-1.5) --(0,0)
(2.5,-.4) node {\text{immediate}}			
(2.5,-.9) node {\text{avaliable pool}}
(4,2) to [R=$R_{3}$,-*] (6,2) -- (6,-3)
			to [C=$C_{3}$,-*] (0,-3) -- (0,-1.5)
(4.5,-2.6) node {\text{small pool}}
(6,2) to [R=$R_{4}$] (8,2)
			to [american current source=Synthesis] (8,-4.5) -- (0,-4.5) -- (0,-3)
(6.5,-4.2) node {\text{large pool}}
;
\end{circuitikz}
\end{center}
\caption{
A circuit realization of the Aristazabal--Glavinovi\v{c} model.}
\end{figure}
For the electric circuit, $E$ is the electromotive source,  $C_i$s are capacities, while $R_i$s are the resistances. Biologically, $E$ represents synthesis,  $C_i$s are the capacities to store vesicles and  $\frac 1{R_{i}C_{j}}$ are interpreted as the pools' replenishment rates.
This results in the following system of ODEs for $\mb U,$
\begin{equation} \label{ag}
\p_t\mb U
=  \mbb K \mb U  + \mb F,
\end{equation}
where
$$ \mbb K = \begin{pmatrix} - \frac 1{R_{1}C_{1}} &  \frac 1{R_{1}C_{2}} & 0 \\
\frac 1{R_{1}C_{1}} & -\frac 1{R_{2}C_{2}}- \frac 1{R_{1}C_{2}}  & \frac 1{R_{2}C_{3}} \\
0 & \frac 1{R_{2}C_{2}} & - \frac 1{R_{2}C_{3}} %- \frac 1{R_{3}C_{3}}
\end{pmatrix}\qquad \mb F = \begin{pmatrix} -\frac 1{R_{0}C_{1}} \\ 0 \\ \frac 1{R_{3}C_{3}}(E- U_3) \end{pmatrix}. $$
\end{example}
\subsubsection{Macro and micro models}
It is worthwhile to reflect on the models discussed in Examples \ref{exRot} and \ref{AG}. Both represent systems that operate and are observable at the macro-scale. This is the scale of our everyday experience as, in principle, we can measure the total number of cells with a particular genetic characteristics or the number of vesicles in a particular state. Such systems, which have characteristics observable at the macro-scale, are called macro-systems. The models introduced above only give heuristic relations between  macro-features, or observables, of such systems, ignoring any underpinning dynamics influencing these relations. Models of this kind, which here we refer to as \emph{macro-models}, are relatively easy to use but often are  too crude. Following the ideas presented in Ref.~ \cite{bel}, we observe that  usually the nodes themselves often have internal structure with a complex dynamics, and the state of the system in these nodes is the  resultant of this dynamics. Moreover, the nodes are interlinked and the rules of connection, given above by the coefficients of the weighted adjacency matrix $\mbb K$, may themselves be determined by dynamical processes taking place along  the edges. These scenarios can occur in various configurations and the models enhanced by considering dynamics at the nodes and/or along the connecting pathways, here called \emph{micro-models}, typically provide a better insight into the dynamics of complex  processes. At the same time, in  such micro-models the added dynamical processes often act at  different time, or size, scales and thus the micro-model actually becomes a  multi-scale model. We observe that the micro-model must be asymptotically consistent with the macro-model; that is, the macroscopic features of the micro-model should be approximately the same as the features provided by solving the macro model. In other words, the observables determined by the solutions of the macro-model must be recoverable from the multi-scale micro-model by taking its  regular or singular limit, and the coefficients of the macro-model should be fully determined by the processes occurring at the micro scale. We note that such a point of view, though in a restricted setting, has been already mathematically explored in \cite{AB12,BB15}.

 It is, however, important to observe that, in general, building a micro-model on the basis a given macro-model may lead to a micro-model whose
 singular limit is completely different from the original macro-model. An example of such a micro-model is offered by the McKendrick model with geographical structure and fast migrations between the patches, described by a matrix $\mbb K$. It follows \cite{BaLabook,BSG,BSG2,BG} that the aggregated dynamics in such a model is given by a scalar McKendrick equation with averaged birth and death coefficients, as explained in details in Example \ref{ex13}, but not by (\ref{ku}).
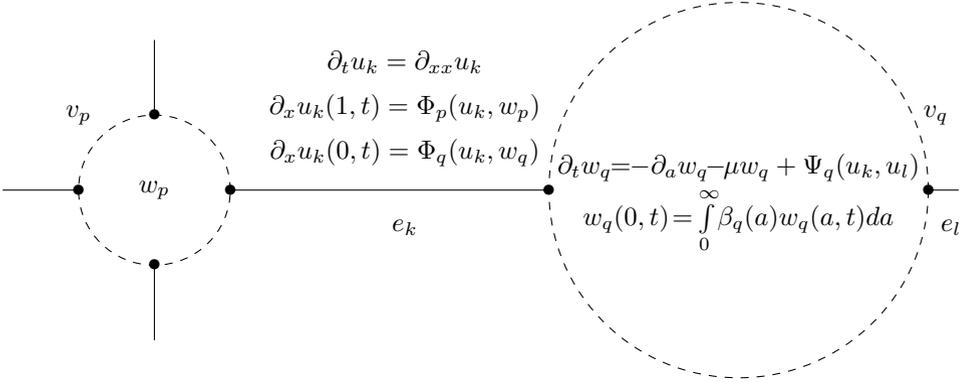
\begin{figure}
\begin{center}
\begin{tikzpicture}
\draw (-4,0) to [,*-*] (.2,0);
\draw	(-5,2) to [,*-*] (-5,1);
\draw	(-7,0) to [,*-*] (-6,0);
\draw	(-5,-1) to [,*-*] (-5,-2);
\draw (5.2,0) to [,*-*] (5.6,0);
%\draw (-5,0) .. controls (-5,1) and (-4,0) .. (-4,1);
\draw[dashed] (-5,0) circle (1cm);
%\draw (-5,0) arc (-45:15:1cm);
%\draw (-5,0) arc (105:165:1cm);
%\draw (-5,0) arc (-160:-100:1cm);
%\draw (-5.5,0.89) to (-5,0);
%\draw (2,-1) arc (-60:90:1cm);
\draw[dashed] (2.7,0) circle (2.5cm);
\draw (-6,0) node {$\bullet$};
\draw (-5,-1) node {$\bullet$};
\draw (-4,0) node {$\bullet$};
\draw (-5,1) node {$\bullet$};
\draw (5.2,0) node {$\bullet$};
\draw (0.2,0) node {$\bullet$};
%\draw (-4.5,0) node {$S_{pk}$};
\draw (2.7,0.3) node {$\p_t w_q\!\!=\!\!-\p_a w_q\!\!-\!\!\mu w_q+\Psi_q(u_{k},u_{l})$};
\draw(2.7,-0.4) node{$w_q(0,t)\!=\!\int\limits_{0}^\infty\! \beta_q(a)w_q(a,t)da$};
\draw (-5,0) node {$w_p$};
\draw (-6,1) node {$v_{p}$};
\draw (5.3,1) node {$v_{q}$};
\draw (-1.7,-.5) node {$e_{k}$};
\draw (5.5,-.5) node {$e_{l}$};
\draw (-1.7,1.7) node {$\partial_{t}u_{k}=\partial_{xx}u_{k}$};
\draw (-1.7,1.1) node {$\p_x u_{k}(1,t)=\Phi_p(u_k,w_p)$};
\draw (-1.7,0.5) node {$\p_x u_{k}(0,t)=\Phi_q(u_k,w_q)$};
%\draw (1,-1.5) node {$\partial_{t}v_{q}(\textbf{x},t)=\partial_{\textbf{xx}}v_{q}(\textbf{x},t)$};
%\draw (1,-2) node {$v_{q}(\textbf{x},0)=g_{q}(\textbf{x})$};
%\draw (1,-2.5) node {$v_{q}|_{\partial\bigcup_{k} S_{qk}}=0$};
%\begin{caption}
%Hypothetical
%\end{caption}
\end{tikzpicture}\end{center}
\caption{Hypothetical micro-model related to a network in which the population in the node $v_q$  evolves according to the McKendrick model influenced by populations migrating through the edges $e_k$ and $e_l$, modelled by a functional $\Psi_q^{k,l}$. On the other hand, the exchange between the nodes occurs through diffusion along the edges with fluxes at the endpoints given by some functionals of the population densities $w_p$ and $w_q$ in the nodes and the density $u_k$ on the edge.}
\label{fig3}
\end{figure}

\section{Main results}
\label{sec2}
Let us consider model  (\ref{ku}), describing the evolution of subpopulations concentrated in $m$ locations, with no internal dynamics, which  influence each other according to the pattern of connections described by the entries of a matrix $\mbb K.$ Our aim is to build a micro-model by including internal dynamics in the subpopulations, which would have asymptotically the same macroscopic characteristics as (\ref{ku}). In this paper we accomplish this by allowing for the subpopulations to evolve along the edges of some (hyper) graph according to either diffusion or transport operators and subject to specific interface conditions at the endpoints of the edges.

Before we formulate the main results, we have to introduce some basic notation. We consider problems in $\mbb R^m$ and the boldface characters will usually denote vectors in $\mbb R^m$, e.g. $\mb u = (u_1,\ldots,u_m).$ We denote  $\mc M = \{1,\ldots,m\}.$ Further, for any Banach space $X$, we will use the notation $\mb X = \underbrace{X\times\ldots\times X}_{m\,times}$, e.g. for $X=L_1(I), I=[0,1]$ we denote $\mb L_1(I)=\underbrace{L_1(I)\times\ldots\times L_1(I)}_{m\,times}$.

The main contribution  of the paper  consists of the following two results.

\paragraph{Diffusion along the edges.}
Let us consider the following initial-boundary problem
\begin{eqnarray}
 \p_t \mb u_\e(x,t) &=& \frac{1}{\e}\p_{xx}\mb u_\e(x,t), \qquad (x,t) \in ]0,1[\times \mbb R_+,\nn\\
 \p_x \mb u_\e(0,t)&=& \e\mbb K^{00}\mb u_\e(0,t) + \e\mbb K^{01}\mb u_\e(1,t),\qquad t>0, \nn\\
 \p_x \mb u_\e(1,t)&=&  \e\mbb K^{10} \mb u_\e(0,t)+\e\mbb K^{11} \mb u_\e(1,t), \qquad t>0,\nn\\
 \mb u_\e(x,0) &=& \mathring{\mb u}(x), \qquad x\in ]0,1[,
 \label{system1'}
\end{eqnarray}
where  $\e>0$ and $\mbb K^{\omega}$, $\omega \in \Omega = \{00, 01,10,11\}$ are real $m\times m$ matrices. Further, let
\ben\label{projection0}
\mb v_\e  = \left(\int_{0}^{1}u_{1,\e}(x)dx, \ldots \int_{0}^{1}u_{m,\e}(x)dx\right)
\een
and let $\bar{\mb v}$ be the solution to (\ref{ku}) with
$$
\mbb K = \mbb K^{10}-\mbb K^{00} +\mbb K^{11}-\mbb K^{01}
$$
and the initial condition given by $\left(\int_{0}^{1}\mathring{u}_{1}(x)dx, \ldots \int_{0}^{1}\mathring{u}_{m}(x)dx\right)$.
Then, for any $\mathring{\mb u}\in  \mb W^2_1(I)$ and $T>0,$ there is $C_T,$ independent of $\e$ and $\mathring{\mb u},$ such that
\begin{equation}
\|\mb v_\e(t) - \bar {\mb v}(t) \|_{\mb X} \leq \e C_T\|\mathring{\mb u}\|_{\mb W^2_1(I)},
\label{finest0}
\end{equation}
uniformly on $[0,T]$, where $\mb X=\mb L_1(I)$ or $\mb X = \mb C(I)$.
\paragraph{Transport along the edges.}

Consider the transport equation
\begin{eqnarray}
\partial_{t}\mathbf{u}_\e(x, t) &=& -\frac{1}{\epsilon}\partial_{x}\mathbf{u}_\e(x, t),\quad x\in ]0, 1[\times \mbb R_+,\nn\\
\mathbf{u}_\e(0, t) &=& \mathbf{u}_\e(1, t) + \epsilon \mathbb{B}\mathbf{u}_\e(1, t),\quad t>0,\nn\\
\mathbf{u}_\e(x, 0)  &=& \mathring{\mathbf{u}}(x),\quad x\in ]0, 1[,
\label{Transport''}
\end{eqnarray}
 where $\e>0$ and $\mbb B$ is an arbitrary matrix. If $\mb v_\e$ is defined by (\ref{projection0}) and $\bar{\mb v}$ is the solution to (\ref{ku}) with $
\mbb K = \mbb B
$
and the initial condition given by $\left(\int_{0}^{1}\mathring{u}_{1}(x)dx, \ldots \int_{0}^{1}\mathring{u}_{m}(x)dx\right)$,
then, for any $\mathring{\mb u}\in  \mb W^1_1(I)$ and $T>0,$ there is $C_T,$ independent of $\e$ and $\mathring{\mb u},$ such that
\begin{equation}
\|\mb v_\e(t) - \bar {\mb v}(t) \|_{\mb L_1(I)} \leq \e C_T\|\mathring{\mb u}\|_{\mb W^1_1(I)},
\label{finest0'}
\end{equation}
uniformly on $[0,T].$

    The results presented here heavily depend on the well-posedness theory of problems of the form (\ref{system1'}) and (\ref{Transport''}) and precise estimates obtained for them. This theory has been developed in Ref.~ \cite{BFN2}. Here we only recall these results when needed.

\begin{remark}
The above results show that the models (\ref{system1'}) and (\ref{Transport''}) both are correct micro-extensions of the macro-model (\ref{ku}) in the sense explained in the introduction; that is, their macroscopic features (here the total size of each subpopulation) evolve approximately in the same way as described by (\ref{ku}).
\end{remark}
\begin{remark}
We note that (\ref{finest0}) and (\ref{finest0'}) fall short of  typical results expected in asymptotic analysis, where one constructs an approximation of the whole solution $\mb u_\e$ by using appropriate initial, and possibly other, layers. This also allows for proving  that the singular limit solution $\bar{\mb v}$ provides a good approximation to $\mb u_\e$ outside a narrow transient layer close to $t=0$, see Ref.~ \cite{BaLabook}. We are able to prove such a result for the diffusion problem. However, for the transport problem the standard initial layer is a fast-oscillating function of zero mean and thus here we cannot claim that the solution to (\ref{ku}) approximates $\mb u_\e$ away from zero.
\end{remark}
\begin{remark}
The  systems of the form (\ref{system1'}) and (\ref{Transport''}) have originated from the modelling of transport and diffusion processes occurring along the edges of a physical graph, where the exchange between subpopulations only could occur at its vertices, see e.g. Refs~ \cite{BaP,AB12,Bres,NagPhysD,Ka,Ko,Kos,Ku1,Ku2,DM}.

To explain the relation between the latter models with (\ref{system1'}) and (\ref{Transport''}), let us recall that by a graph we understand the pair  $ \mc G = (V,  E)$, where, say, $ V$ is the set of $n$ vertices and $ E$ is the set of $m$ edges.  Note that an edge here is an (unordered) pair of vertices.  In transport problems, since transport, in contrast to diffusion, has a direction, it is more convenient to impose an orientation on each edge and work with a directed graph. A directed graph, or digraph, is the pair $G =(V(G), E(G)) = (\{v_1,\ldots, v_n\}, \{e_1,\ldots,e_m\})$, where, in contrast to $\mc G$, the edges here are defined as ordered pairs of vertices.
An important role is played by the line digraph $L(G)$ of $G.$ To recall
$L(G) = (V(L(G)), E(L(G))) = (E(G), E(L(G))),$ where
$$
E(L(G)) =\{uv;\; u,v \in E(G), \text{the\;head\;of\;}u\;\text{coincides\;with\;the\;tail\;of\;}v\}.
$$
An analogous definition can be made for the line graph $L(\mc G)$ of the graph $\mc G$. To avoid considerations related to the geometry of the edges, see e.g. Ref.~ \cite[Section 2.2.1]{DM}, from the beginning we assume that each edge is identified with $[0,1].$

Then, under some technical conditions \cite{BaP,KS04}, the transport of a substance along the edges of $G$ with Kirchhoff's type interface conditions at the vertices can be cast in the form of (\ref{Transport''}) with some matrix $\mbb A$ describing the boundary condition (in (\ref{Transport''}), $\mbb A = \mbb I +\e\mbb B$). However, in such a case,  $\mbb A$ must be the weighted (transposed) adjacency matrix of $L(G)$, see Ref.~ \cite{BF1}.

Similarly, the diffusion along the edges of some graph $\mc G$ with Robin type interface conditions \cite{AB12} can be written in the form (\ref{system1'}) but then the matrices $\mbb K^\omega$, $\omega =\{00,01,10,11\},$ must satisfy certain conditions which allow for the construction of the (weighted) adjacency matrix of the line graph of $\mc G,$ see Ref.~ \cite{BF1}. A more detailed presentation of this case is given in Example \ref{ex33}.

It is easy to see that not every matrix $\mbb A$ is a weighted adjacency matrix of a line graph. At the same time, there are legitimate models of the form (\ref{Transport''}) with arbitrary (nonnegative) matrix in the boundary conditions. An example is rendered e.g. by the discrete Rotenberg-Rubinov-Lebowitz model \cite{rot}, where the subpopulations can communicate without any physical connection between them.

We emphasize that the results presented in this paper are valid for arbitrary matrices in the boundary conditions of (\ref{system1'}) and (\ref{Transport''}). However,  in many cases it is important to determine whether there is a physical graph structure behind the model. This question was addressed in Ref.~ \cite{BF1}.

We note that another approach would be to consider hypergraphs as any graph is the line graph of a hypergraph \cite{Be}, but we will not pursue this line of research here.
\end{remark}

 \section{Diffusion problems}

We shall consider (\ref{system1'})  in both $\mb C(I)$ and $\mb L_1(I)$. For technical reasons, we also shall need the space $\mb W^1_1(I)$.  The norms in these spaces will be denoted, respectively, by $\|\cdot\|_\infty, \|\cdot\|_0, \|\cdot\|_1$ but, if it does not lead to any misunderstanding, we  use $\mb X$ to denote any of these spaces and $\|\cdot\|$ or $\|\cdot\|_\mb X$ to denote the norm in $\mb X$. Further, by $|||\cdot|||$ with appropriate subscript we denote the operator norm in respective $\mb X$.

To keep the notation in line with Ref.~ \cite{BFN2} (and with Ref.~ \cite{Gre}, on which the well-posedness results are based), we introduce operator $L$ by the formula
\begin{equation}
\mb X \ni \mb u \to L\mb u = (\gamma_0 \p_x\mb u, \gamma_1 \p_x \mb u) \in \mb Y = \mbb R^m\times \mbb R^m,
\label{Phi}
\end{equation}
where $\gamma_i, i=0,1,$ is the trace operator at $x=i$ (taking the value at $x=i$ if $\mbb X = \mb C(I)$). The domains of $L$ are $D(L) = \mb C^1(I)$ if $\mb X =\mb C(I)$ and $D(L) = \mb W^2_1(I)$ in the other cases. Then we define the operator
$$
\mb X\ni \mb u \to \Phi\mb u = \mbb K(\gamma_0 \mb u, \gamma_1\mb u) = \left(\begin{array}{cc}\mbb K^{00}& \mbb K^{01}\\\mbb K^{10}& \mbb K^{11}\end{array}\right)\left(\begin{array}{c}\gamma_0 \mb u\\\gamma_1 \mb u\end{array}\right) \in \mbb R^m\times \mbb R^m.
$$
Let $\mathsf A$ denote the differential expression $\mathsf A\mb u := \p_{xx}\mb u.$ Then we define the operators $\mb A^\alpha_\Phi$, $\alpha =\infty, 0, 1,$ by the restriction of $\mathsf A$ to the domains
\begin{eqnarray*}
D(\mb A^\infty_\Phi) &=& \{\mb u \in \mb C^2(I);\;L\mb u = \Phi \mb u\},\\
D(\mb A^0_\Phi) &=& \{\mb u \in \mb W^2_1(I);\;L\mb u = \Phi \mb u\},\\
D(\mb A^1_\Phi) &=& \{\mb u \in \mb W^3_1(I);\;L\mb u = \Phi \mb u\},
\end{eqnarray*}
respectively. If the base space is clear from the context, we shall drop the superscript from the notation.
Further, let us denote
\begin{equation}
\Phi^*\mb u = \mbb K^*(\gamma_0 \mb u, \gamma_1\mb u) = \left(\begin{array}{cc}\mbb K^{00\,T}& -\mbb K^{10\,T}\\-\mbb K^{01\,T}& \mbb K^{11\,T}\end{array}\right)\left(\begin{array}{c}\gamma_0 \mb u\\\gamma_1 \mb u\end{array}\right),
\label{phistar}
\end{equation}
where $\mbb K^T$ denotes the transpose of $\mbb K$. Clearly
\begin{equation}
(\Phi^*)^* = \Phi.
\label{phistar'}
\end{equation}
In general, if $\mb A$ is the generator of a semigroup,  we denote by \sem{\mb A} the semigroup generated by $\mb A$.
\subsection{Basic estimates}\label{ss31}
Consider the resolvent equation for (\ref{system1'}) with $\e=1,$
\begin{equation}
\la \mb u - \p_{xx} \mb u = \mb f, \quad x \in ]0,1[.
\label{reseq1}
\end{equation}
Its solution $\mb u$ is given by
\begin{equation}
\mb u(x) = \mb C_1e^{-\mu x} +\mb C_2 e^{\mu x} + \mb U_\mu(x)
\label{gensol1}
\end{equation}
where $0\neq \la = \mu^2 = |\la|e^{i\theta}$ with $\Re \mu >0,$
\begin{equation}
\mb U_\mu(x) = \frac{1}{2\mu} \cl{0}{1} e^{-\mu|x-s|}\mb f(s)ds.
\label{Umu}
\end{equation}
and  $\mb C_1$ and $\mb C_2$ are determined by the boundary conditions. Further, denote
$$
\Sigma_{\alpha,\theta_0}= \{\la = |\la|e^{i\theta}\in \mbb C;\;|\la|\geq \alpha, |\theta|\leq \theta_0 <\pi\}$$
and $ \Sigma_{0, \theta_0} = \Sigma_{\theta_0}.$

 We note that $\mb A_0$  corresponds to the Neumann boundary conditions in one of the spaces $\mb X = \mb C(I), \mb L_1(I), \mb W^1_1(I).$ Standard calculations \cite{EN,BFN2} give
  \begin{equation}
\|R(\la, \mb A_0)f\|_{\mb X} \leq \frac{2\|f\|_{\mb X}}{|\la|\cos \theta_0/2}, \qquad \la \in \Sigma_{\theta_0}
\label{resw11}
\end{equation}
for any $\theta_0<\pi$. Hence, in particular, $\mb A_0$ generates an analytic semigroup in $\mb X$.

With regards to (\ref{system1'}) in $\mb X=\mb C(I), \mb L_1(I), \mb W^1_1(I)$, we have 
\begin{theorem}\cite{BFN2}
 Then there is $\alpha\geq 0$ such that for any $\theta_0<\pi$
\begin{equation}
\|R(\la, \mb A_\Phi)\|_{\mb X} \leq  2\|R(\la, \mb A_0)\|_{\mb X}, \qquad \la \in  \Sigma_{\alpha, \theta_0}.
\label{grei0}
\end{equation}
Hence, $\mb A_\Phi$ generates an analytic semigroup in $\mb X$. Moreover, \sem{\mb A_\Phi} is positive if and only if
$-\mbb K^{00},  \mbb K^{11}$ are nonnegative off-diagonal and $-\mbb K^{01}, \mbb K^{10}$ are nonnegative.
\label{thgre}\end{theorem}
For further use we need some intermediate results from the proof. First, since in $\mb X=\mb C(I), \mb W^1_1(I),$ the operator $\Phi$ in the boundary condition $L\mb u = \Phi\mb u$ is compact (and thus also bounded), Theorem \ref{thgre} is a straightforward consequence of Theorem 2.4 of Ref.~ \cite{Gre}. More precisely, it follows that
\begin{equation}
R(\la, \mb A_\Phi) = (I-L_\la\Phi)^{-1}R(\la, \mb A_0),
\label{grei1}
\end{equation}
where $L_\la = (L|_{Ker (\la - \mb A)})^{-1}, \la \in \rho(\mb A_0)$ and the proof relies on the fact that if $\Phi$ is compact and $\mb A_0$ generates an analytic semigroup, then there is $\alpha\geq 0$ such that for
 $\la \in \Sigma_{\alpha,\theta_0}$ we have $|||L_\la\Phi|||_{\mb X}\leq 1/2$  which, by the Neumann series, gives (\ref{grei0}).

The situation with the generation in $\mb X = \mb L_1(I)$ is more complicated as $\Phi$ is not bounded, and so cannot be compact, on $\mb X$. A more involved argument shows that the resolvent $R(\la, A^0_\Phi)$ exists in the same sector as $R(\la, A^1_\Phi)$ and, moreover, $(R(\la, A^\infty_{\Phi^*}))^*\mb f = R(\bar \la, A^1_\Phi)\mb f$ whenever $f\in \mb L_1(I)$ (the adjoint of   $R(\la, A^\infty_{\Phi^*})$ acts in the space of complex vector measures and here we identify $\mb L_1(I)$ with the subspace of (densities of) absolutely continuous measures). Then, since the norm of an absolutely continuous measure equals the $L_1$ norm of its density \cite{Bobks}, we arrive at
\begin{equation}
\|R(\la, \mb A^0_\Phi)\|_{\mb L_1(I)} = \|R(\bar\la, \mb A^\infty_{\Phi^*})\|_{\mb C(I)}.
\label{resnorm}
\end{equation}
Being densely defined,  $\mb A^0_\Phi$  generates an analytic semigroup on $\mb L_1(I)$.

\subsection{A lifting theorem}
In asymptotic analysis we need results on solvability of the inhomogeneous problem
\begin{eqnarray}
 \p_t \mb u(x,t) &=& d\p_{xx}\mb u(x,t) +\mb f(x,t), \qquad (x,t) \in ]0,1[\times \mbb R_+,\nn\\
 \p_x \mb u(0,t)&=& \mbb K^{00}\mb u(0,t) + \mbb K^{01}\mb u(1,t) + \bphi_0(t),\qquad t>0, \nn\\
 \p_x \mb u(1,t)&=&  \mbb K^{10} \mb u(0,t)+\mbb K^{11} \mb u(1,t) + \bphi_1(t), \qquad t>0,\nn\\
 \mb u(x,0) &=& \mathring{\mb u}(x), \qquad x\in ]0,1[,
 \label{sys2}
\end{eqnarray}
where $\mb f, \bphi_0$ and $\bphi_1$ are known functions and $d>0$.  First, we observe that  $\mb v(x) = -x(1-x)((\bal+\bbet)x -\bal)$ satisfies
\begin{eqnarray}
 \p_x \mb v(0)&=& \mbb K^{00}\mb v(0) + \mbb K^{01}\mb v(1) + \bal, \nn\\
 \p_x \mb v(1)&=&  \mbb K^{10} \mb v(0)+\mbb K^{11} \mb v(1) + \bbet.\label{lftth}
 \end{eqnarray}\label{lift}
 Thus, if $\mb v$ is the function constructed above with $\bal = \bphi_0$ and $\bbet = \bphi_1$ then, by substitution $\mb U = \mb u -\mb v,$  problem (\ref{sys2}) is reduced to
\begin{eqnarray}
 \p_t \mb U(x,t) &=& d\p_{xx}\mb U(x,t) +\mb f(x,t) + 2d((\bphi_0(t) +\bphi_1(t))x -2\bphi_0(t)-\bphi_1(t)) \nn\\
 &&- x^3(\p_t\bphi_0(t) +\p_t\bphi_1(t)) -x^2(2\p_t\bphi_0(t)+\p_t\bphi_1(t)) + x\p_t\bphi_0(t) ,\nn\\
 \p_x \mb U(0,t)&=& \mbb K^{00}\mb U(0,t) + \mbb K^{01}\mb U(1,t) \nn\\
 \p_x \mb U(1,t)&=&  \mbb K^{10} \mb U(0,t)+\mbb K^{11} \mb U(1,t), \qquad t>0,\nn\\
 \mb U(x,0) &=& \mathring{\mb u}(x)+x(1-x)((\bphi_0(0)+\bphi_1(0))x -\bphi_0(0)),
 \label{sys3}
\end{eqnarray}
  which is classically solvable, \cite{Pa}, on $]0,T[$ for some $T>0$ provided e.g. the inhomogeneity is locally H\"{o}lder continuous on $]0,T[$.

\subsection{Basic results on solvability of (\ref{system1'})}

Let us return to (\ref{system1'}).  Let $\e>0$.
 We denote by $\mb A_{\e,\Phi}$ the operator given by the restriction of the expression $\mb u \to \mathsf A_\e \mb u= \e^{-1}\p_{xx}\mb u$ to the domain
$
D(\mb A_{\e,\Phi}) = \{\mathbf{u}\in \mb X;\; \p_{xx}\mb u \in \mb X,   L\mb u = \e \Phi \mb u\},
$
where $\mb X = \mb C(I)$ or $\mb X = \mb L_1(I)$. It is easy to see, by (\ref{grei0}) and (\ref{resnorm}),  that there exists a positive analytic semigroup \sem{\mb A_{\e,\Phi}} in $\mb X$ that solves (\ref{system1'}). We show uniform boundeness of \sem{\mb A_{\e,\Phi}} with respect to $\e$.
\begin{lemma} For any $\mb X=\mb C(I), \mb W^1_1(I), \mb L_1(I),$ there are constants $M_\phi$ and $\omega_\phi$, independent of $\e$, such that
\begin{equation}
\|e^{t\mb A_{\e,\Phi}}\|_{\mc L(\mb X)} \leq M_\Phi e^{\omega_\Phi t}, \qquad t\geq 0.
\label{type}
\end{equation}
\end{lemma}
\begin{proof} First we observe that since the resolvent equation
$
\la \mb u - \e^{-1} \p_{xx}\mb u = \mb f
$
can be written as $
{\e\la} \mb u - \p_{xx}\mb u = \e{\mb f}
$
and the Neumann boundary condition does not involve  $\e$,  estimate (\ref{resw11}) is uniform in $\e$; that is, for $\mb X=\mb C(I), \mb L_1(I), \mb W^1_1(I)$,
\begin{equation}
\|R(\la, \e^{-1}A_0)\mb f\|_{\mb X} \leq \frac{2\|\mb f\|_{\mb X}}{|\la|\cos \theta_0/2}, \qquad \la \in \Sigma_{\theta_0}.
\label{resw11e}
\end{equation}
Then, as indicated below (\ref{grei1}),  the generation result in $\mb X =\mb C(I)$ and in $\mb X=\mb W^1_1(I)$ follows due to possibility of establishing the estimate $|||L_\la\Phi|||_{\mb X} \leq 1/2$ in $\Sigma_{\alpha, \theta_0}$.  Here we have to show that for (\ref{system1'}) such an estimate is valid uniformly in $\e$.

In (\ref{system1'}) we deal with the family of operators $\{\e^{-1}\mb A\}_{\e>0}$, hence we have to introduce the family $L_{\la, \e} = (L|_{Ker (\la - \e^{-1}\mb A)})^{-1}$. Since $Ker (\la-\e^{-1}\mb A) = Ker (\e\la-\mb A),$ we have $L_{\la,\e} = L_{\e\la, 1} = L_{\e\la}$, where the latter refers to the case when the diffusion coefficients equals 1, see Subsection \ref{ss31}.  We see that  $Ker (\la-\e^{-1}\mb A)$ consists of functions given by
$
\mb u(x) = \mb C_1e^{\mu x} + \mb C_2e^{-\mu x},
$  where $\mu^2 = \e\la$.
Hence $L\mb u = (\mu \mb C_1-\mu \mb C_2, \mu \mb C_1e^\mu -\mu \mb C_2e^{-\mu})$. Thus, for a given $(\mb d_1, \mb d_2)\in \mbb R^m\times \mbb R^m$ we have
$$
L_{\e\la}(\mb d_1,\mb d_2) = \frac{\mb d_2-\mb d_1e^{-\mu }}{\mu(e^{\mu}-e^{-\mu})}e^{\mu x} + \frac{\mb d_1e^{\mu }-\mb d_2}{\mu(e^{\mu}-e^{-\mu})}e^{-\mu x}.
$$
Then, for $\mb X=\mb L_1(I)$ we have
\begin{eqnarray*}
&&\|L_{\e\la}(\mb d_1,\mb d_2)\|_{\mb X}\\
 &&\leq \frac{e^{-\Re\mu}(\|\mb d_2\|+\|\mb d_1\|e^{-\Re\mu })(e^{\Re\mu}\!\!-\!\!1)}{|\mu|\Re\mu(1-\!\!e^{-2\Re\mu})} +
\frac{e^{-\Re\mu}(\|\mb d_2\|+\|\mb d_1\|e^{\Re\mu })(1-e^{-\Re\mu})}{|\mu|\Re\mu(1\!\!-e^{-2\Re\mu})}\\
&&\leq \frac{\|\mb d_2\|+\|\mb d_1\|e^{-\Re\mu }}{|\mu|\Re\mu(1+e^{-\Re\mu})} +
\frac{\|\mb d_2\|e^{-\Re\mu }+\|\mb d_1\|}{|\mu|\Re\mu(1+e^{-\Re\mu})}\leq \frac{C}{\e|\la|\cos\theta_0/2}\|(\mb d_1,\mb d_2)\|,
\end{eqnarray*}
where $C$ is independent of $\e$ and we used the estimate
\begin{equation}
\left |\frac{1}{e^\mu-e^{-\mu}}\right| = e^{-\Re \mu}\left|\sum\limits_{n=0}^{\infty}e^{-2n\mu}\right| \leq
e^{-\Re \mu}\sum\limits_{n=0}^{\infty}e^{-2n\Re \mu} = \frac{e^{-\Re \mu}}{1-e^{-2\Re\mu}}.
\label{estemu}
\end{equation}
For $\mb X=\mb W^1_1(I)$ we additionally need to estimate $\p_xL_{\e\la}(\mb d_1,\mb d_2)$. Since the differentiation only introduces the multiplier $\pm\mu$ in each term, we have
$$
\|\p_xL_{\e\la}(\mb d_1,\mb d_2)\|_{\mb L_1(I)}\leq \frac{C}{\sqrt{\e|\la|}\cos\theta_0/2}\|(\mb d_1,\mb d_2)\|.
$$
Hence, noting that in (\ref{grei1}) the operator $\Phi$ is replaced by $\e\Phi$ and $(\mb d_1,\mb d_2) = \e\Phi\mb u$ with $\|(\mb d_1,\mb d_2)\|\leq \e M_0\|\mb u\|_{\mb W^1_1(I)}$ if $\mb u \in \mb W^1_1(I),$ with $M_0$ independent of $\e$, we obtain
\begin{equation}
\|\e L_{\la,\e}\Phi\mb u\|_{\mb W^1_1(I)} \leq \frac{M_1}{\sqrt{|\la|}} \|\mb u\|_{\mb W^1_1(I)},
\label{ela}
\end{equation}
where $M_1$ is a constant independent of $\e$ and $\la\in \Sigma_{\alpha, \theta_0}$ some fixed $\alpha>0$ .

For $\mb C(I),$ we observe that the norm in $\mb W^1_1(I)$ is stronger that that in $\mb C(I)$ and also $\|(\mb d_1,\mb d_2)\|\leq \e M'\|\mb u\|_{\mb C(I)}$ if $\mb u \in \mb C(I),$ with  $M'$ independent of $\e$. Hence
\begin{equation}
\|\e L_{\la,\e}\Phi\mb u\|_{\mb C(I)}\leq C'\|\e L_{\la,\e}\Phi\mb u\|_{\mb W^1_1(I)} \leq \frac{M_2}{\sqrt{|\la|}} \|\mb u\|_{\mb C(I)},
\label{ela'}
\end{equation}
where $M_2$ is independent of $\e$ and $\|\la\|$ and $C'$ is the embedding constant.

Finally, the result for $\mb L_1(I)$ follows from (\ref{resnorm}) and (\ref{ela'}). Thus (\ref{type}) holds. \end{proof}

\subsection{Asymptotic state lumping}
\subsubsection{Formal expansion}\label{ssFE}
We follow the standard asymptotic analysis approach, \cite{BaLabook}. First, we find the hydrodynamic space $\mb V$ of $\mb A_{\e,\Phi}$; that is, the null space of $\mb u\to \mbb \p_{xx}\mb u$ on $D(\mb A_{0,0}) = D(\mb A_{0})$. This amounts to finding the solution of $m$ uncoupled Neumann problems
\begin{equation}
\p_{xx}u_{i} = 0,\qquad \p_x u_i(0) = \p_x u_i(1) =0.
\label{npr}
\end{equation}
Hence,  $\mb V = \mathrm{ span}\{\mb e_i\}_{i\in \mc M}$,
where $\mathbf{e}_{i}$ are versors of $\mbb R^m$.  The (formal) spectral projection onto $\mb V$ is given by
\ben\label{projection}
\mb v = \mathcal{P}\mathbf{u} = \left(\int_{0}^{1}u_{1}(x)dx, \ldots \int_{0}^{1}u_{m}(x)dx\right),
\een
see (\ref{projection0}).  Hence we decompose  \begin{equation}\mathbf{u}_{\epsilon}(x, t) = [\mathcal{P}\mathbf{u}_{\epsilon}]( t) + [\mc Q\mathbf{u}_{\epsilon}](x, t) =:  \mathbf{v}_{\epsilon}(t) + \mathbf{w}_{\epsilon}(x, t),\label{pq}\end{equation}
 where $\mc Q = I - \mathcal{P}$.  Further, to shorten notation,  we denote
\begin{eqnarray*}
&&\mbb K =  \mbb K^{10}-\mbb K^{00} +\mbb K^{11}-\mbb K^{01},\\
&&\mbb K^1_+ = \mbb K^{10}+\mbb K^{11},\qquad \mbb K^0_+ = \mbb K^{01}+\mbb K^{00},\\
&&\mbb K^1_- = \mbb K^{11}-\mbb K^{01},\qquad \mbb K^0_- = \mbb K^{10}-\mbb K^{00}.
\end{eqnarray*}
 If $\mb u_\e$ is a classical solution to (\ref{system1'}), we can apply $\mathcal{P}$ to get
\begin{eqnarray}
\partial_t \mathbf{v}_\e(x,t) \!\!\!&=& \!\!\! \frac{1}{\epsilon}\mathcal{P}\partial_{xx}\mathbf{u}_\e(x,t) \!= \! \frac{1}{\epsilon}\left(\int_{0}^{1}\partial_{xx}u_{\e,1}(x, t)dx, \ldots, \int_{0}^{1}\partial_{xx}u_{\e, m}(x, t)dx\right)\nn\\
& =& \!\!\!\frac{1}{\epsilon} \left(\partial_{x}\mathbf{u}_\e(1, t) - \partial_{x}\mathbf{u}_\e(0, t)\right)\nn\\
&=&\!\!\! (\mbb K^{10}-\mbb K^{00})\mb u_\e(0,t) + (\mbb K^{11}-\mbb K^{01}) \mb u_\e(1,t) \nn\\
&=& \!\!\!\mbb K \mb v_\e(t) + \mbb K_-^{0}\mb w_\e(0,t) +\mbb K_-^{1} \mb w_\e(1,t).\label{veq}
\end{eqnarray}
Similarly, applying $\mathcal{Q} = (I - \mathcal{P})$ to (\ref{system1'}), we get
\begin{eqnarray}
 \p_t \mb w_\e (x,t)&=& \partial_{t}(\mathbf{u}_\e(x,t)-\mb v_\e(x,t)) = \frac{1}{\epsilon} \partial_{xx} (\mathbf{v}_{\epsilon}(x,t)+ \mathbf{w}_{\epsilon}(x,t)) - \p_t\mb v_\e(x,t)\nn\\
& =& \frac{1}{\epsilon}\partial_{xx} \mathbf{w}_{\epsilon} (x,t)-\mbb K \mb v_\e(t) - \mbb K_-^{0}\mb w_\e(0,t) -\mbb K_-^{1} \mb w_\e(1,t).\label{weq}
\end{eqnarray}
The boundary conditions for $\mb w_\e$ at $x=1$ are found from
\begin{eqnarray}
\left.\partial_{x}\mathbf{w}_{\epsilon}(x,t)\right\vert_{x=1} &=&  \left.\partial_{x}\mathbf{u}_{\epsilon}(x,t)\right\vert_{x = 1} = \epsilon \mbb K^{11}\mathbf{u}_{\epsilon}(1, t) + \epsilon \mbb K^{10} \mathbf{u}_{\epsilon}(0, t)\nn\\
& = &\epsilon \mbb K^{1}_+\mathbf{v}_{\epsilon}(t) + \epsilon \mbb K^{11}\mathbf{w}_{\epsilon}(1, t) + \epsilon \mbb K^{10}\mathbf{w}_{\epsilon}(0, t).
\label{wx1}
\end{eqnarray}
Similarly, at $x=0$ we have
\begin{equation}
\left.\partial_{x}\mathbf{w}_{\epsilon}(x,t)\right\vert_{x=0} = \epsilon \mbb K^{0}_+\mathbf{v}_{\epsilon}(t) + \epsilon \mbb K^{01}\mathbf{w}_{\epsilon}(1, t) + \epsilon \mbb K^{00}\mathbf{w}_{\epsilon}(0, t).
\label{wx0}
\end{equation}
Now, expanding $\mathbf{w}_{\epsilon} = \bar{\mb w}_0 +\e\bar{ \mb w}_1 + O(\e^2),$ and substituting it into (\ref{weq})--(\ref{wx0}), we find,  for $\epsilon^{0},$
\begin{equation}
\p_{xx}\bar{\mathbf{w}}_{0}  = 0 \quad \p_x\bar{\mathbf{w}}_{0}(1, t)  = \p_x\bar{\mathbf{w}}_{0}(0, t)  = 0.
\label{EEq6}
\end{equation}
From this equation we conclude that $\mathbf{w}_{0}$ is constant. Since, by (\ref{pq}), $\mathbf{w}_{\epsilon}$ and thus all terms of its expansion  annihilate constants, we conclude that $\bar{\mathbf{w}}_0 = 0$. Hence the terms containing $\mb w_\e$ in (\ref{veq}) are of higher order in $\e$ and we can consider the following limit equation for $\mb v_\e$
\begin{equation}
\p_t{\bar{\mb v}} = \mbb K \bar{\mb v},\quad  \bar{\mb v}(0) = \mc P\!\init{\mb u}.
 \label{barv}
 \end{equation}
Next, at the  $\epsilon$ level, we have
\begin{eqnarray}
&&\p_{xx}\bar{\mathbf{w}}_{1} = \mbb K\bar{\mathbf{v}}(t),\nn \\
&&\p_x\bar{\mathbf{w}}_{1}(1, t) = \mbb K^1_+ \bar{\mathbf{v}}(t), \qquad
\p_x\bar{\mathbf{w}}_{1}(0, t) = \mbb K^0_+ \bar{\mathbf{v}}(t),
\label{w1}
\end{eqnarray}
where $\bar{\mathbf{v}}$ is given by (\ref{barv}). Integrating, using the boundary conditions and the definition of $\mb W$, we find
\begin{equation}
\bar{\mb w}_1(x,t) = \frac{1}{2}x^2 \mbb K\bar{\mathbf{v}}(t) + x \mbb K^0_+ \bar{\mathbf{v}}(t) + \mb d(t)
\label{w1sol}
\end{equation}
with
\begin{equation}
\mb d(t) = -\left(\frac{1}{3}\mbb K^0_+  + \frac{1}{6} \mbb K^1_+\right)\bar{\mb v}(t).
\label{w1sol1}
\end{equation}
Clearly, the pair $(\bar {\mb v}, \e\bar {\mb w}_1)$ cannot be an approximation of $\mb u_\e = (\mb v_\e, \mb w_\e)$ unless $\mb w_\e|_{t=0} = O(\e)$. To improve the approximation we introduce the initial layer by rescaling time as $\tau = t/ \epsilon$ and expand  both $\tilde{\mathbf{v}}(x, \tau)$ and $\tilde{\mathbf{w}}(x,\tau)$  in $\epsilon:$
$$
\tilde{\mathbf{v}}(x,\tau) = \tilde{\mathbf{v}}_0(x,\tau)  + O(\e),
\qquad \tilde{\mathbf{w}}(x,\tau) = \tilde{\mathbf{w}}_0(x,\tau) + O(\e).$$
  Standard calculations give $\tilde{\mathbf{v}}_0 = 0$ and, using  (\ref{weq}), at the $\e^0$ level, we get
\begin{eqnarray}
\partial_{\tau} \tilde{\mathbf{w}}_0(x,\tau) & =& \partial_{xx}\tilde{\mathbf{w}}_0(x,\tau), \quad
\partial_{x}\tilde{\mathbf{w}}_0(1, \tau) = \partial_{x}\tilde{\mathbf{w}}_0(0, \tau) = 0,\nn\\
\ti{\mb w}_0(x,0)&=& \mathring{\mb w}(x) = \mathring{\mb u}(x)-\mc P\!\mathring{\mb u}.
\label{Fourier}
\end{eqnarray}
Since $\int_0^1 \ti{\mb w}_0(x,0)dx =0$, the solution can be written as the Fourier series
\ben\label{F-solution}
\tilde{\mathbf{w}}_0(x,\tau) =  \sum_{n = 1}^{\infty}e^{\lambda_n\tau}\mathbf{a}_{n}\cos(n\pi x),
\een
where $\lambda_n =- (n \pi)^2 < 0,$ $n=1,\ldots,$  and
\begin{equation}
\mb a_n = 2\cl{0}{1}\mathring{\mb w}(x) \cos(n\pi x) dx.
\label{an}
\end{equation}
Now, we can formulate the main theorem of this section.
\begin{theorem}
Let $\mb X=\mb C(I), \mb L_1(I)$. Let $\mb u_\e(t) = e^{t\mb A_{\e, \Phi}}\mathring{\mb u}$ with $\mathring{\mb u}\in \mb W^2_1(I)$ be the solution of (\ref{system1'}) in $\mb X$,  $\bar{\mb v}$ be the solution to (\ref{barv}) and $\ti{\mb w}_0$ be the solution to (\ref{Fourier}). Then, for any $0<T<\infty,$ there is  $C= C(T, \mbb K^{00}, \mbb K^{01}, \mbb K^{10}, \mbb K^{11})$ such that
\begin{equation}
\|\mb u_\e(t) - \bar {\mb v}(t) - \ti{\mb w}_0(t/\e)\|_{\mb X} \leq \e C\|\mathring{\mb u}\|_{\mb W^2_1(I)}
\label{finest}
\end{equation}
uniformly on $[0,T]$.
Furthermore, if $\mathring{\mb u}\in \mb X$, then we have
\begin{equation}
 \lim\limits_{\e\to 0}\|\mb u_\e(t) - \bar{\mb v}(t) - \ti{\mb w}_0(t/\e)\|_{\mb X} =0
 \label{weak}
 \end{equation}
 uniformly on $[0,T]$.
\label{mainth1}
\end{theorem}
\subsubsection{Proof of Theorem \ref{mainth1}}
The first problem we encounter is that since $\mathring{\bf w}$ does not satisfy the boundary conditions of (\ref{Fourier}) and thus it is not in the domain of the generator, $\ti{\mb w}_0$ is not a classical solution to (\ref{Fourier}) up to $t=0$. To remedy this, we need the following lemma.
\begin{lem}
Let $\mb u \in \mb W^2_1(I)$ and $\mb X =\mb L_1(I)$ or $\mb X = \mb C(I)$. Then for any $\delta>0$ there is $\mb u_\delta \in \mb W^2_1(I)$ satisfying $\p_x\mb u_\delta(0) = \p_x\mb u_\delta(1) =0$ and such that
\begin{eqnarray}
\|\mb u -\mb u_\delta\|_{\mb X} &\leq& C_1\delta \|\mb u\|_{\mb W^2_1(I)},\label{u1}\\
\|\p_x\mb u_\delta\|_{\mb X} &\leq& C_2 \|\mb u\|_{\mb W^2_1(I)},\label{u2}\\
\|\p_{xx}\mb u_\delta\|_{\mb L_1(I)} &\leq& C_3 \|\mb u\|_{\mb W^2_1(I)},\label{u3}\\
 \sum_{n = 1}^{\infty}|\mathbf{a}_{n,\delta}| &\leq& C_4 \|\mb u\|_{\mb W^2_1(I)},
 \label{udelta}
 \end{eqnarray}
 where $C_i, i=1,2,3$ are independent of $\delta$ and $\mb a_{n,\delta}$ are coefficients (\ref{an})  for $\mb u_\delta$.
\label{udeltalem}
\end{lem}
\begin{proof}
Let $\mb u \in \mb W^2_1(I)$. Then $\p_x\mb u$ absolutely continuous and, for any $x\in I,$ we have $|\mb u(x)|\leq C\|\mb u\|_{\mb W^2_1(I)}$ and
$|\p_x\mb u(x)|\leq C'\|\mb u\|_{\mb W^2_1(I)}.$ For $0<\delta<1/2,$ we define
\begin{equation}
\mb u_\delta(x) = \left\{\begin{array}{lcl}\mb u(x)&\mathrm{for}& x \in [\delta, 1-\delta],\\
\frac{\p_x\mb u(\delta)}{2\delta}(x^2-\delta^2) + \mb u(\delta) &\mathrm{for}& x \in [0,\delta[,\\
\frac{\p_x\mb u(1-\delta)}{2\delta}(\delta^2-(1-x)^2) + \mb u(1-\delta) &\mathrm{for}& x \in ]1-\delta,1].
\end{array}
\right.
\end{equation}
We see that
\begin{equation}
\p_x\mb u_\delta(x) = \left\{\begin{array}{lcl}\p_x\mb u(x)&\mathrm{for}& x \in [\delta, 1-\delta],\\
\frac{\p_x\mb u(\delta)}{\delta}x  &\mathrm{for}& x \in [0,\delta[,\\
\frac{\p_x\mb u(1-\delta)}{\delta}(1-x)  &\mathrm{for}& x \in ]1-\delta,1]
\end{array}
\right.\label{pxu}
\end{equation}
and, since $\p_x\mb u_\delta$ is continuous at $x=\delta, 1-\delta$,
\begin{equation}
\p_{xx}\mb u_\delta(x) = \left\{\begin{array}{lcl}\p_{xx}\mb u(x)&\mathrm{for}& x \in [\delta, 1-\delta],\\
\frac{\p_x\mb u(\delta)}{\delta} &\mathrm{for}& x \in [0,\delta[,\\
\frac{-\p_x\mb u(1-\delta)}{\delta}  &\mathrm{for}& x \in ]1-\delta,1]
\end{array}
\right.\label{pxxu}
\end{equation}
hence $\mb u_\delta\in \mb W^2_1(I)$.
Now, we have
\begin{eqnarray*}
\|\mb u -\mb u_\delta\|_{\mb L_1(I)} &\leq&  \cl{0}{\delta}|\mb u(x)|dx + \frac{\p_x\mb u(\delta)}{2\delta}\cl{0}{\delta} |x^2-\delta^2|dx
+ \mb u(\delta)\cl{0}{\delta} dx\\
&&+ \frac{\p_x\mb u(1-\delta)}{2\delta}\!\!\cl{1-\delta}{1} |(1-x)^2-\delta^2|dx
+ \mb u(1-\delta)\!\!\cl{1-\delta}{1} dx \\
&&\leq C_1\delta \|\mb u\|_{\mb W^2_1(I)}
\end{eqnarray*}
and
\begin{eqnarray*}
\|\mb u -\mb u_\delta\|_{\mb C(I)} &\leq&  \sup\limits_{x\in [0,\delta]}|\mb u(x)-\mb u(\delta)| + \sup\limits_{x\in [0,\delta]}\left|\frac{\p_x\mb u(\delta)}{2\delta}(x^2-\delta^2)\right|\\
 &&+\sup\limits_{x\in [1-\delta,1]}|\mb u(x)-\mb u(1-\delta)| \\
 &&+ \sup\limits_{x\in [1-\delta,1]}\left|\frac{\p_x\mb u(1-\delta)}{2\delta}((1-x)^2-\delta^2)\right|\leq C_1\delta\|\mb u\|_{W^2_1(I)},
\end{eqnarray*}
where we used $|\mb u(x)-\mb u(y)|\leq \sup_{\xi \in I}|\p_x \mb u(\xi)||x-y| \leq \|\mb u\|_{W^2_1(I)}|x-y|$, $x,y \in I$. This gives (\ref{u1}).
Analogous calculations, using (\ref{pxu}) and (\ref{pxxu}) give  (\ref{u2}) and (\ref{u3}).

Finally, integrating twice by parts and using the boundary conditions yields
\begin{eqnarray*}
\frac{1}{2}\mb a_{n,\delta} %&=& \cl{0}{1}\mb u_{\delta}(x) \cos n\pi x\, dx = \frac{1}{\pi n}\mb u_{\delta}(x)\sin n\pi x|^{x=1}_{x=0} - \frac{1}{\pi n}\cl{0}{1} \p_x\mb u_\delta(x)\sin n\pi x\, dx\\
&=& \frac{1}{\pi^2 n^2}\p_x\mb u_{\delta}(x)\cos n\pi x|^{x=1}_{x=0} - \frac{1}{\pi^2 n^2}\cl{0}{1} \p_{xx}\mb u_\delta(x)\cos n\pi x\, dx.
\end{eqnarray*}
Hence
$$
|\mb a_{n,\delta}| \leq  \frac{2}{\pi^2 n^2}\cl{0}{1} |\p_{xx}\mb u_\delta(x)|\, dx
 $$
 and thus (\ref{udelta}) holds.

\end{proof}

Further, we observe that if $\mb u_\delta$ is such an approximation of $\mb u,$ then $\mb w = \mb u- \mc P \mb u$, where $\mc P$ is defined by (\ref{projection}), can be approximated in the same way by $\mb w_\delta = \mb u_\delta - \mc P\mb u_\delta$. Indeed, since $\mc P\mb u_\delta$ is a constant, $\mb w_\delta$ is in $\mb W^2_1(I)$, $\p_x \mb w_\delta = \p_x\mb u_\delta$ so that the boundary conditions are satisfied and
\begin{equation}
\mb w - \mb w_\delta = \mb u-\mb u_\delta - \cl{0}{1}(\mb u(x)-\mb u_\delta(x))dx
\label{wdelta}
\end{equation}
so that $\mb w_\delta$ approximates $\mb w$ in each $\mb X$:
$$
\|\mb w - \mb w_\delta\|_{\mb X} \leq \|\mb u-\mb u_\delta\|_{\mb X} + \|\mb u-\mb u_\delta\|_{\mb L_1(I)} \leq 2C_1\delta\|\mb u\|_{\mb W^2_1(I)}.
$$
In what follows, we shall use the initial layer $\ti{\mb w}_{0,\delta}$ which is the solution to (\ref{Fourier}) with the initial condition $\mathring{\mb w}_\delta$ defined as $\mathring {\mb w}_\delta = \mathring{\mb u}_\delta - \mc P\mathring{\mb u}_\delta$.
We approximate  $(\mathbf{v}_\e, \mathbf{w}_\e)$ by $(\bar{\mathbf{v}}, \epsilon\mathbf{w}_{1} +\ti {\mb w}_{0,\delta})$ and find that the  error $
\mb {e_v} = \mb v_\e -\bar{\mb v}$, $\mb {e_w} = \mb w_\e - \e \bar{\mb w}_1 -\ti{ \mb w}_{0,\delta}$ formally satisfies, by (\ref{barv}), (\ref{w1}) and and (\ref{Fourier}), the equations
\begin{eqnarray}
\p_t \mathbf{e}_{\mathbf{v}}(t) &=& \mbb K\mathbf{e}_{\mathbf{v}}(t) + \mbb K^1_- (\mathbf{e}_{\mathbf{w}}(1, t) + \epsilon \bar{\mathbf{w}}_{1}(1, t) + \ti{\mb w}_{0,\delta}(1,\tau)) \nn\\&& + \mbb K^0_-(\mathbf{e}_{\mathbf{w}}(0, t)+ \epsilon \mathbf{w}_{1}(0, t)+\ti{\mb w}_{0,\delta}(0,\tau)),\nn\\
\partial_t \mathbf{e}_{\mathbf{w}}(x,t) &=&  \frac{1}{\epsilon}\p_{xx}\mathbf{e}_{\mathbf{w}}(x,t) -\mbb K\mathbf{e}_{\mathbf{v}}(t)\nn\\&&   -\mbb K_-^1(\mathbf{e}_{\mathbf{w}}(1,t)+ \epsilon \mathbf{w}_{1}(1,t)+\ti {\mb w}_{0,\delta}(1,\tau)) \nn\\&&-\mbb K_-^0(\mathbf{e}_{\mathbf{w}}(0,t) + \epsilon \mathbf{w}_{1}(0,t)+\ti {\mb w}_{0,\delta}(0,\tau)) - \epsilon \partial_t\bar{\mathbf{w}}_1,\nn\\
\mb{e_v}(0) &=& \mb v_\e(x,0)-\bar{\mb v}(x,0) =0,\nn\\
\mathbf{e}_{\mathbf{w}}(x, 0) &=&- \epsilon \bar{\mathbf{w}}_1(x,0) +\mathring{\mb w}(x)- \mathring{\mb w}_\delta(0).
\label{eveq}
\end{eqnarray}
For the boundary conditions,  using (\ref{wx1}), (\ref{w1}) and (\ref{Fourier}) we have
\bdm
\bad
\partial_{x}\mathbf{e}_{\mathbf{w}}(1,t) &= \partial_{x}\mathbf{w}_\e(1, t) - \epsilon \partial_{x}\bar{\mathbf{w}}_1(1, t)-\p_x\ti{\mb w}_{0,,\delta}(1,\tau)\\
& = \epsilon \left(\mbb K^1_+\mathbf{e}_{\mathbf{v}}(t)+ \mbb K^{11}\mathbf{e}_{\mathbf{w}}(1, t) +\mbb K^{10}\mathbf{e}_{\mathbf{w}}(0, t)\right) \\&\phantom{x}+
\e (\mbb K^{11}\tilde{\mathbf{w}}_{0,\delta}(1, \tau) +\mbb K^{10}\tilde{\mathbf{w}}_{0,\delta}(0, \tau)) \\
&\phantom{x}+\e^2 (\mbb K^{11}\bar{\mathbf{w}}_1(1, \tau) +\mbb K^{10}\bar{\mathbf{w}}_1(0, \tau)).
\ead
\edm
Similarly, using (\ref{wx0}) and again (\ref{Fourier}), we obtain
\bdm
\bad
\partial_{x}\mathbf{e}_{\mathbf{w}}(0,t) &= \partial_{x}\mathbf{w}_\e(0, t) - \epsilon \partial_{x}\bar{\mathbf{w}}_1(0, t)-\p_x\ti{\mb w}_{0,\delta}(x,\tau)\\
& = \epsilon \left(\mbb K^0_+\mathbf{e}_{\mathbf{v}}(t)+ \mbb K^{01}\mathbf{e}_{\mathbf{w}}(1, t) +\mbb K^{00}\mathbf{e}_{\mathbf{w}}(0, t)\right)\\
&\phantom{x} +
\e (\mbb K^{01}\tilde{\mathbf{w}}_{0,\delta}(1, \tau) +\mbb K^{00}\tilde{\mathbf{w}}_{0,\delta}(0, \tau))\\& \phantom{x} +\e^2 (\mbb K^{01}\bar{\mathbf{w}}_1(1, \tau) +\mbb K^{00}\bar{\mathbf{w}}_1(0, \tau)).
\ead
\edm
To proceed, we observe that, since the semigroup  \sem{\mb A_{\e,\Phi}} is analytic, for a given $\mathring{\mb u}\in \mb X$ the function $t\to \mb u_\e(t) = e^{t\mb A_{\e,\Phi}}\mathring{\mb u}$ is a classical solution of the problem in the sense that for any $t>0$ it is infinitely (classically) differentiable and the system of differential equations and the boundary conditions are pointwise satisfied.  Similarly, all other terms of the asymptotic expansion are sufficiently regular for all terms in the above equations to be well defined. Hence, denoting $\mb E = \mb e_{\mb v}+ \mb e_{\mb w}$, adding the equations for $\mb e_\mb v$ and $\mb e_\mb w$ and using the regularity of each term, we see that $\mb E$ is the classical solution of the problem
\begin{eqnarray}
\p_t \mb E &=& \frac{1}{\e} \p_{xx} \mb E - \e\p_t\bar{\mb w}_1,\nn\\
\p_x\mb E|_{x=0} &=& \e \mbb K^{01}\mb E|_{x=1}+ \e \mbb K^{00}\mb E|_{x=0} + \e (\bPsi_{0,\delta} +\e\bPhi_0),\nn\\
\p_x\mb E|_{x=1} &=& \e \mbb K^{11}\mb E|_{x=1}+ \e \mbb K^{10}\mb E|_{x=0} + \e (\bPsi_{1,\delta} +\e\bPhi_1),\nn\\
\mb E( 0) &=&- \epsilon \bar{\mathbf{w}}_1(0) +\mathring{\mb w}- \mathring{\mb w}_\delta,\label{53}
\end{eqnarray}
where
\begin{eqnarray*}
\bPsi_{0,\delta} +\e\bPhi_0 &=&  \mbb K^{01}\tilde{\mathbf{w}}_{0,\delta}|_{x=1} +\mbb K^{00}\tilde{\mathbf{w}}_{0,\delta}|_{x=0}+\e (\mbb K^{01}\bar{\mathbf{w}}_1|_{x=1} +\mbb K^{00}\bar{\mathbf{w}}_1|_{x=0}),\nn\\
\bPsi_{1,\delta} +\e\bPhi_1 &=& \mbb K^{11}\tilde{\mathbf{w}}_{0,\delta}|_{x=1} +\mbb K^{10}\tilde{\mathbf{w}}_{0,\delta}|_{x=0} +\e (\mbb K^{11}\bar{\mathbf{w}}_1|_{x=1} +\mbb K^{10}\bar{\mathbf{w}}_1|_{x=0}).
\end{eqnarray*}
 Next, by (\ref{sys3}), we see that the substitution \begin{eqnarray*}\mb F(x,t) &=&\mb E(x,t) +\e x(1-x)((\bPsi_{0,\delta}(\tau) +\bPsi_{1,\delta}(\tau) +\e(\bPhi_0(t)+\bPhi_1(t)))x \\
 &&-\bPsi_{\delta}(\tau)-\e\bPhi_0(t))\end{eqnarray*} reduces (\ref{53}) to
\begin{eqnarray}
\p_t \mb F &=& \frac{1}{\e} \p_{xx} \mb F +  \e \bar {\mb G} + \tilde{
\mb G}_\delta,\nn\\
\p_x\mb F|_{x=0} &=& \e \mbb K^{01}\mb F|_{x=1}+ \e \mbb K^{00}\mb F|_{x=0},  \nn\\
\p_x\mb F|_{x=1} &=& \e \mbb K^{11}\mb F|_{x=1}+ \e \mbb K^{10}\mb F|_{x=0},\nn\\
\mb F(x, 0) &=& \mathring{\mb w}(x)- \mathring{\mb w}_\delta(x) +\epsilon ( -\bar{\mathbf{w}}_1(x,0)+ \mb H_\delta(x,0)),\label{Feq}
\end{eqnarray}
where, since $\bar {\mb w}_1(0,0) = -\left(\frac{1}{3}\mbb  K^0_++\frac{1}{6}\mbb K^1_+\right)\!\!\init{\mb v}$ and
$\bar {\mb w}_1(1,0) = \left(\frac{1}{3}\mbb  K^1_++\frac{1}{6}\mbb  K^0_+\right)\!\!\init{\mb v}$, we have
\begin{eqnarray*}
\bar {\mb G}(x,t) &=&-\p_t\bar{\mb w}_1(x,t)  +2  (\bPhi_0(t)+\bPhi_1(t))x - 2\bPhi_0(t)-\bPhi_1(t),\\
\ti {\mb G}_\delta(x,\tau)&=& \p_\tau\left(-x^3 (\bPsi_{0,\delta}(\tau)+\bPsi_{1,\delta}(\tau))\right.\\
 &&-\left. x^2(2 \bPsi_{0,\delta}(\tau)+\bPsi_{1,\delta}(\tau)) + x\bPsi_{0,\delta}(\tau)\right),\\
\mb H_\delta(x,0)&=& x(1-x)\left(\left(\mbb  K^{1}_+\!\!\init{\mb w}_\delta\!\!(1) + \mbb  K^{0}_+\!\!\init{\mb w}_\delta\!\!(0)\right.\right.\\
 &&\left.+\frac{\e}{3} \!\init{\mb v}\!\!\left.(\mbb K^{01}+ \mbb K^{11})\left(\mbb  K^1_++2\mbb  K^0_+\right) - (\mbb K^{00}+\mbb K^{10})\left(\mbb  K^0_++2\mbb  K^1_+\right)\phantom{\frac{\e}{3}}\!\!\!\!\!\right)x \right.\\
 && - \!\left.\left(\mbb  K^{01}\!\!\init{\mb w}_\delta\!\!(1) +\mbb  K^{00}\!\!\init{\mb w}_\delta\!\!(0)\right.\right.\\
  &&\left.\left.+\frac{\e}{3} \left (\mbb K^{01}\!\!\left(\mbb  K^1_++2\mbb  K^0_+\right) - \mbb K^{00}\!\!\left(\mbb  K^0_++2\mbb  K^1_+\right)\right)\!\init{\mb v}\right)\!\right).
\end{eqnarray*}
We observe that $\mb H_\delta$ is well defined as  for $\tau>0$
$$
\tilde{\mathbf{w}}_{0,\delta}(0,\tau) =  \sum_{n = 1}^{\infty}e^{\lambda_n\tau}\mathbf{a}_{n,\delta},\qquad
\tilde{\mathbf{w}}_{0,\delta}(1,\tau) =  \sum_{n = 1}^{\infty}(-1)^ne^{\lambda_n\tau}\mathbf{a}_{n,\delta}
$$
and, by Lemma \ref{udeltalem}, both series are uniformly convergent on $[0,T], T<\infty,$ and thus define continuous functions on $[0,T].$ Thus
the values $\tilde{\mathbf{w}}_{0,\delta}(0,0), \tilde{\mathbf{w}}_{0,\delta}(1,\tau)$ are well defined and equal to $\mathring{\mb w}_\delta(0)$ and $\mathring{\mb w}_\delta(1)$, respectively.

Similarly, we observe that for $\tau>0$
$$
\p_\tau\Psi_{0,\delta}(\tau) = \mbb K^{01}\sum_{n = 1}^{\infty}(-1)^n\la_n e^{\lambda_n\tau}\mathbf{a}_{n,\delta} +\mbb K^{00}\sum_{n = 1}^{\infty}\la_ne^{\lambda_n\tau}\mathbf{a}_{n,\delta}
$$
and, by (\ref{udelta}),
$$
\cl{0}{\infty}|\p_\tau\Psi_{0,\delta}(\tau)|dt \leq C\cl{0}{\infty}\sum\limits_{n = 1}^{\infty}\pi n e^{-\pi n\frac{t}{\e}}|\mathbf{a}_{n,\delta}|dt \leq \e C\sum\limits_{n=1}^\infty |\mathbf{a}_{n,\delta}|
\leq\e C_4 \|\mb u\|_{\mb W^2_1(I)}.
$$
In the same way
$$
\cl{0}{\infty}|\p_\tau\Psi_{1,\delta}(\tau)|dt  \leq \e C_4 \|\mb u\|_{\mb W^2_1(I)}.
$$
Then, using the mild formulation of the solution to (\ref{Feq}), we obtain on $[0,T]$
\begin{eqnarray*}
&&\|\mb F(t)\|_{\mb X} \leq \|e^{t\mb A_{\e,\Phi}}(\mathring{\mb w}-\mathring{\mb w}_\delta + \e \mb H_\delta)\|_{\mb X}\\
&&+ \|\cl{0}{t}e^{(t-s)\mb A_{\e,\Phi}}(\e \bar {\mb G}(s) + \ti{\mb G}(s/\e))ds\|_{\mb X} \\
&\leq &\!M_\Phi e^{\omega_\Phi T}\!\!\left(\!\|\mathring{\mb w}-\mathring{\mb w}_\delta\|_{\mb X}\!\! +\e\|\mb H_\delta\|_{\mb X}\!\! +\! \e\!\! \cl{0}{T} \!\|\bar{ \mb G}(s)\|_{\mb X}ds + \! \cl{0}{T}\!\|\ti{ \mb G}_\delta(s/\e)\|_{\mb X}ds\!\right)\\
&\leq& M_\Phi e^{\omega_\Phi T} \left(\|\mathring{\mb w}-\mathring{\mb w}_\delta\|_{\mb X}+ \e C'\|\mathring{\mb u}\|_{\mb W^2_1(I)}\right)
\end{eqnarray*}
for some constant $C'$. Using the fact that $C'$ is independent of $\delta$ and the smallness of the auxiliary terms, we see  that for any $T<\infty$ there is $C$ such that
$$
 \|\mb u_\e(t) - \bar{\mb v}(t) - \ti{\mb w}_0(t/\e)\|_{\mb X} = \e C\|\mathring{\mb u}\|_{\mb W^2_1(I)
}$$
uniformly on $[0,T]$, where  $\mathring{\mb u} \in W^2_1(I)$ is arbitrary. Using the density of $\mb W^2_1(I)$ in $\mb X,$ uniform boundedness of \sem{\mb A_{\e,\Phi}} on $[0,T]$ and corollary to the Banach-Steinhaus theorem ($3-\e$ lemma), we also obtain (\ref{weak})  for any $\mathring{\mb u} \in \mb X$.

\begin{exa}\label{ex33}
It is interesting to consider to which extent the properties of the original system are inherited by the limit one.  For this we return to the  neurotransmitter model of Aristazabal and Glavinovi\v{c}, described in Example \ref{AG}. This macro-model was enhanced to a micro-model along the lines described in the Introduction in Refs~ \cite{BoMo,AB12}, where the authors corrected the first attempt of such a construction developed in Ref.~ \cite{BK}. In the micro-model it is assumed that the pools with vesicles occupy some physical space and the vesicles in each pool move according to a diffusion process and migrate between the pools by crossing semipermeable membranes separating the pools.
 In Ref.~ \cite{BoMo}, the authors interpreted the pools as the edges of a finite graph connected by the nodes at which the exchange of vesicles takes place. In this particular model the graph is linear, however, in Ref.~ \cite{AB12} the author generalized the model so that it was posed on
  a finite graph without loops $\mc G = (\mc V, \mc E)$ with, say $n$ vertices and $m$ edges.  On each edge there is a substance with density $u_j, j\in \mc M,$ which diffuses along this edge and can also enter the adjacent edges across the vertices that join them according to a version of the Fick law.   To write down its analytical form, first we note that, since diffusion does not have a preferred direction, we can assign the tail, or the left endpoint, (that is, 0)  and  the head, or the right endpoint (that is, 1)  to the endpoints of the edge in an arbitrary way. Let $l_j$ and $r_j$ be the rates at which the substance leaves $e_j$ through, respectively, the left and the right endpoints and  $l_{jk}$ and $r_{jk}$ be the rates of at which it subsequently enters the edge $e_k$.
%\ben\label{trans1}
%\sum_{i\neq j}{l_{ij}} \leq l_i, \quad \sum_{i\neq j}{r_{ij}} \leq r_i.
%\een
Then the Fick laws at, respectively,  the head and the tail of $e_i$  are given by
\begin{equation}
-\p_x u_{i}(1) = r_{i}u_{i}(1) - \sum_{j\neq i}{r_{ij}u_{j}(v)},\quad
\p_x u_{i}(0) =l_iu_i(0) - \sum_{j\neq i}{l_{ij}u_{j}(v)},
\label{trans3}
\end{equation}
 where we used $u_j(v)$ as $v$ may be either the tail or the head of  $e_j$. If there are no edges incident to $e_i$, then respective coefficients $r_{ij}$ or $l_{ij}$ are equal to 0.

It is clear that if $r_{ij}\neq 0,$ then $l_{ij} =0$ and if $l_{ij}\neq 0,$ then $r_{ij} =0$ and thus we can define an $m\times m$ matrix $\mbb A$ by setting $a_{ij}=1$ if either $r_{ij}=1$ or $l_{ij} =1$ and zero otherwise. Then $\mbb A$ is the adjacency matrix of the line graph $L(\mc G)$ of $\mc G$. Such a matrix is, however, not easy to use and we see that introducing, for any $i,j \in \mc M,$
\begin{eqnarray}
k^{00}_{ij} &=& -l_{ij} \quad \mathrm{if}\; v = 0,\qquad    k^{01}_{ij} = -l_{ij}\quad  \mathrm{if}\; v=1, \qquad k^{00}_{ii} = l_{i},\nn\\
     k^{10}_{ij} &=& r_{ij}\quad \mathrm{if}\; v = 0,\qquad k^{11}_{ij} = r_{ij} \quad \mathrm{if}\; v=1, \qquad k^{11}_{ii} = -r_{i},
      \label{ktor}
      \end{eqnarray}
      where $v$ is either the tail or the head of the edge under consideration, the boundary conditions can be written as in (\ref{system1'}).
 In Ref.~ \cite{AB12} the author considered the expected position of the particle on the graph which  led to a Feller process and thus the problem was posed in the space of continuous functions. On the other hand, in Ref.~ \cite{AG13} the probability density of the process was considered and this resulted in a Markov process posed in a space of integrable functions. The corresponding diffusion problems are adjoint to each other, as in (\ref{phistar}) and (\ref{resnorm}).
 Certainly, this interpretation requires in both cases the solution to be   a scalar function and thus we identify $\mb C(I)$ (resp. $\mb L_1(I)$) with the space $C(\mb I)$ (resp. $L_1(\mb I)$), where $\mb I = [0_1,1_1]\cup\ldots\cup [0_m,1_m]$; that is, instead of considering a vector function on $I$ we consider a scalar function on a disconnected compact space composed of $m$ disjoint closed intervals, see \cite{BFN2}. In particular, each edge $e_j, j\in \mc M,$ is identified with the closed interval $[0_j,1_j].$ Then  $\mb A_\Phi$ will be changed to $A_\Phi$ which is  the restriction of $$
(Au)(x) = \sum\limits_{j\in \mc M}\chi_{[0_j,1_j]}(x)\sigma_j\p_{xx}u(x)$$
to the space of $C^2$ (rep. $W^2_1$) functions on $Int\, \mb I = ]0_1,1_1[\cup\ldots\cup ]0_m,1_m[$ which, in the former case, are $C^1$ on $\mb I$ and satisfy
$$
\p_xu(0_j) = \sum\limits_{k=1}^m k^{00}_{jk}u(0_k) + \sum\limits_{k=1}^m k^{01}_{jk}u(1_k),\quad
\p_xu(1_j) = \sum\limits_{k=1}^m k^{10}_{jk}u(0_k) + \sum\limits_{k=1}^m k^{11}_{jk}u(1_k).
$$
 Denote by \sem{A_{\e,\Phi}} such a realization of \sem{\mb A_{\e,\Phi}} in $\mb X=C(\mb I)$ or $\mb X= L_1(\mb I)$. Recalling the relation between the model of Refs~ \cite{AB12,AG13} and our formulation, let $\Phi$ correspond to the matrices $\mbb K^\omega$, $\omega \in\{00,01,10,11\},$ according to (\ref{ktor}). Then, by Ref.~ \cite{BFN2},  both \sem{A_{\e,\Phi}} and  its adjoint \sem{A_{\e, \Phi^*}}, see (\ref{phistar}), are positive, in $C(\mb I)$ and $L_1(\mb I)$, respectively, for any $\epsilon$. Let us concentrate on the process in $L_1(\mb I)$.  It is clear that it is a Markov semigroup if and only if
\begin{eqnarray*}
0&=&\p_t \left(\sum\limits_{j\in \mc M}\cl{\mb I}{}u(x)dx )= \sum\limits_{j \in \mc M} (\p_xu(1_j)-\p_xu(0_j) \right)\\
&=& \sum\limits_{i=1}^m \left(-u(0_i) \sum\limits_{j=1}^m (k^{01}_{ij}+k^{00}_{ij})+ u(1_i)\sum\limits_{j=1}^m (k^{11}_{ij}+k^{01}_{ij})\right).
\end{eqnarray*}
As the end-point values are arbitrary,  \sem{A_{\e,\Phi^*}} is Markov if and only if
\begin{equation}
\sum\limits_{j=1}^m (k^{01}_{ij}+k^{00}_{ij}) =0, \qquad \sum\limits_{j=1}^m (k^{11}_{ij}+k^{01}_{ij})=0, \quad i\in \mc M,
 \label{stoch}
 \end{equation}
  which, according to (\ref{ktor}) gives $l_i = \sum_j l_{ij}$ and $r_i = \sum_j r_{ij}$, as stated in Ref.~ \cite{AB12}.

If we consider this formulation in the approximation (\ref{barv})
$$
\p_t{\bar{\mb v}} = \mbb K \bar{\mb v},\quad  \bar{\mb v}(0) = \mc P \mathring{\mb u},
$$
then $$\mbb K = -\mbb K^{10\,T}-\mbb K^{00\,T}+\mbb K^{11\,T}+\mbb K^{01\,T}.$$ By (\ref{ktor}), $\mbb K$ is positive off-diagonal and, by (\ref{stoch}) it is a Kolmogorov matrix and thus (\ref{barv}) represents a continuous time Markov chain in which the original distributed states are `lumped' together at the vertices.
\end{exa}

\section{Transport problems}
\subsection{Solvability of the general transport problem}
In the same way as with diffusion, we begin by considering (\ref{Transport''}) with $\e=1$ and denote then $\mbb T = \mbb I +\mbb B.$ The considerations of this section, however, are valid in a more general setting \cite{BFN2}. Hence, consider
\begin{equation}
\mb u_t(x,t) = [\mb A\mb u](x,t), \qquad\mb u(0,t) = \mbb T\mb u(1,t) ,\quad \mb u(x,0) = \mathring{\mb u}(x),
\label{ACP1b}
\end{equation}
where $\mb A$ is the realization of the expression  $\mathsf A\mb u= -\p_x \mb u$ on the domain $
D(\mb A) = \{\mb u \in \mb W^{1}_1(I);\; \mb u(0) = \mbb T\mb u(1)\}.$
   We emphasize that $\mbb T$ is an arbitrary (not necessarily nonnegative) matrix. For simplicity we only consider this problem in $\mb X = \mb L_1(I)$, since in $\mb C(I)$ the domain $D(\mb A)$ is not dense. Then we have the following result \cite{BFN2}.
\begin{theorem} The operator $(\mb A, D(\mb A))$ generates a $C_0$-semigroup on $\mb L_1(I)$. The semigroup is positive if and only if $\mbb K\geq 0$.\label{thtragen}
\end{theorem}
Denote by \sem{\mb A_{\mbb T}} this semigroup. Then,  \cite{BaP,KS04},
 \begin{equation}
 [e^{t\mb A_{\mbb T}}\mathring{\mb u}](x) = \mbb T^n \mathring{\mb u}(n+x-t),\qquad -n\leq x-t\leq-n+1.
 \label{repr1}
 \end{equation}

\subsection{Asymptotic state lumping in transport problems}

It is easy to see that integrating (\ref{ACP1b}) and denoting $$\mb v(t)  = \left(\int_{0}^{1}u_{1}(x,t)dx, \ldots \int_{0}^{1}u_{m}(x,t)dx\right),$$ as in (\ref{projection}),  we obtain the system
$$
\p_t\mb v(t) + (\mbb I-\mbb T)\mb u(1,t) = 0.
$$
If one could assume that $\mb u$ slowly varies along each edge; that is, $\mb u(x,t) \approx const(t) = \mb u(1,t)$, thus giving $\mb v(t)=\int_0^1\mb u(x,t)dx\approx \mb u(1,t)$, then $\mb v$ could be approximated by the solution $\bar{\mb v}$ of the system
$$
\p_t\bar{\mb v} = (\mbb T-\mbb I)\bar{\mb v},
$$
as in the approximating equation (\ref{barv}) for the diffusion equation. In the latter case the above assumption was physically plausible as diffusion tends to flatten fluctuations of the density and, with fast diffusion, the density quickly becomes homogeneous in space. With transport, the situation is not so obvious as one of the characteristics of transport processes is that the initial profile propagates in time with little or no distortion and the choice of a suitable scaling is not obvious. In Section \ref{pitfal} we will show some cases which do not yield the expected result but here we will continue with a properly rescaled model.

\subsubsection{Properties of solutions to (\ref{Transport''})}

For each $\e>0$ we define an operator $(\mb A_{\epsilon}, D(\mb A_{\epsilon}))$ as the realization of the expression ${\sf A}_\e\mb u
= -\frac{1}{\epsilon}\partial_x \mathbf{u}$ on the domain
$
D(\mb A_{\epsilon})  := \{\mathbf{u} \in \mb W_{1}^{1}(I) ;\;\quad \mathbf{u}(0) = \mathbf{u}(1) +\epsilon \mathbb{B}\mathbf{u}(1)\}
$
and the limit operator $(\mb A, D(\mb A))$ as ${\sf A} \mathbf{u}  := -\partial_x \mathbf{u}$ restricted to
$
D(\mb A)  := \{\mathbf{u} \in \mb W_{1}^{1}(I);\;  \mathbf{u}(0) = \mathbf{u}(1)\}.
$
 Notice that, as for the diffusion,  each operator from $\mb A_{\epsilon}, {\epsilon \geq 0},$ is defined on a different domain. By Theorem \ref{thtragen}, for each $\e> 0$ there exists a semigroup \sem{\mb A_\e}.  We also have the `limit' semigroup \sem{\mb A}. As in the diffusion case, we show that \sem{\mb A_\e} are bounded uniformly in $\e$.

 \begin{lemma} For every $\e_0>0$ there is a constant $M$ such that for any $\e\in ]0,\e_0]$ and $t\in \mbb R_+$ we have
 \begin{equation}
 \|e^{\mb A_{\epsilon} t}\mathring{\mathbf{u}}\|_{\mb X} \leq Me^{t\|\mbb B\|}\|\mathring{\mathbf{u}}\|_{\mb X}.
 \label{unibd}
 \end{equation}\label{uneps}
 \end{lemma}
 \begin{proof}
 In our case, (\ref{repr1}) becomes
 \bd
[e^{\mbb A_{\epsilon} t}\mathring{\mathbf{u}}](x)=(\mathbb{I}+\epsilon \mathbb{B})^{n}\mathring{\mathbf{u}}\left(n+x-\frac{1}{\epsilon}t\right) \qquad \text{for} \quad -n\leq x-\frac{1}{\epsilon}t\leq-n+1.
\ed
Hence, for $n-1 \leq t/\e\leq n$ we have, by changing variable of integration,
\begin{eqnarray}
&&\left\|e^{\mb A_{\epsilon} t}\mathring{\mathbf{u}}\right\|_{\mb X}\leq\|(\mathbb{I}+\epsilon \mathbb{B})^{n}\|\!\!\!\!\!\!\!\cl{0}{\frac{t}{\epsilon}-n+1}\!\!\!\!\!\left\|\mathring{\mb u}\left(n+x-\frac{t}{\epsilon}\right)\right\|dx]\nn\\&&\phantom{xxxxxxxxxx}+
\|(\mathbb{I}+\epsilon \mathbb{B})^{n-1}\|\!\!\!\!\!\!\!\cl{\frac{t}{\epsilon}-n+1}{1}\!\!\!\!\!\left\|\mathring{\mb u}\left(n-1+x-\frac{t}{\epsilon}\right)\right\|dx \nn\\
%&&\leq\|(\mathbb{I}+\epsilon |\mathbb{B}|)^{n}\|\!\!\!\!\cl{n-\frac{t}{\epsilon}}{1}\!\!\!\left\|\mathring{\mb u}\left(z\right)\right\|dz+
%\|(\mathbb{I}+\epsilon |\mathbb{B}|)^{n-1}\|\!\!\!\!\cl{0}{n-\frac{t}{\epsilon}}\!\!\!\!\left\|\mathring{\mb u}\left(z\right)\right\|dz \nn\\
&&\phantom{xxxxxxxx}\leq
(1+\epsilon\left\| \mathbb{B}\right\|)^{n}\left\|\mathring{\mathbf{u}}\right\|_{\mb X}.\label{estimp}
\end{eqnarray}
As $n\leq \frac{t}{\epsilon}+1$,
\bd
(1+\epsilon\left\| \mathbb{B}\right\|)^{n}\left\|\mathring{\mathbf{u}}\right\|_{\mb X} \leq((1+\epsilon\left\| \mathbb{B}\right\|)^{\frac{1}{\epsilon\left\| \mathbb{B}\right\|}})^{\|\mathbb{B}\|t}(1+\epsilon\left\| \mathbb{B}\right\|)\left\|\mathring{\mathbf{u}}\right\|_{\mb X} \leq Me^{t\|\mbb B\|}\|\mathring{\mathbf{u}}\|_{\mb X}\text{.}
\ed

\end{proof}
\subsubsection{Formal expansion}
Formal steps are similar to that in Section \ref{ssFE}, so that we only provide a brief summary of them. First, the hydrodynamic space $\mb V$ consists of solutions to
    $$
    \p_x\mathbf{u} = 0,\qquad
     \mb u(0,t) = \mb u(1,t),
     $$
 so that $\mb V$ is spanned by $\{\mb e_i\}_{i\in \mc M},$ where $\mb e_j$, $j\in \mc M,$ are versors of $\mbb R^m$.
Clearly, the formal adjoint problem is given by
$$
{\partial}_{x}{\bphi}=0,\quad {\bphi}(0, t) = \bphi(1, t),
$$
and thus the projection onto $\mb V$ is given as in (\ref{projection0}); that is, by
\be
\mathcal{P}\mathbf{u}=\left(\cl{0}{1}{u}_{1}(x,t)dx,\cl{0}{1}{u}_{2}(x,t)dx,...,\cl{0}{1}{u}_{m}(x,t)dx\right).
\ee
Projecting the solution of (\ref{Transport}) onto the hydrodynamic and kinetic subspaces we have $\mathbf{u}_{\epsilon} = \mathcal{P}\mathbf{u}_{\epsilon} + \mathcal{Q}\mathbf{u}_{\epsilon}=\mathbf{v}_{\epsilon}+\mathbf{w}_{\epsilon},$ where $\mc Q=\mc I-\mc P$ so that $\mb w_\e \in \mb W =\{\mb w\in \mb X;\; \int_{0}^{1}\mb w(x)dx =0\}.$ Projecting the equations onto $\mb V$ and $\mb W$, we get
\begin{equation}
\p_t\mathbf{v}_{\epsilon}(t) = \mathbb{B}\mathbf{u}(0, t) = \mathbb{B}(\mathbf{v}_{\epsilon}(t) + \mathbf{w}_{\epsilon}(0, t)),\quad
\mathbf{v}_{\epsilon}(0)  = \mathcal{P}\mathring{\mathbf{u}}\label{Picar}
\end{equation}
and
\begin{eqnarray}
\partial_t \mathbf{w}_{\epsilon}(x,t) &=& -\frac{1}{\epsilon}\partial_{x}\mathbf{w}_{\epsilon}(x,t) - \mathbb{B}\mathbf{v}_{\epsilon}(x,t) - \mathbb{B}\mathbf{w}_{\epsilon}(0, t),\nn\\
\mathbf{w}_{\epsilon}(0, t) &=& \mathbf{w}_{\epsilon}(1, t) + \epsilon \mathbb{B}\mathbf{w}_{\epsilon}(1, t) + \epsilon \mathbb{B}\mathbf{v}_{\epsilon}(t),\nn\\
\mathbf{w}_{\epsilon}(x, 0) &=& \mathring{\mathbf{u}}(x) - \mathcal{P}\mathring{\mathbf{u}}\text{.}\label{perTransport}
\end{eqnarray}
Substituting the expansion $
\mathbf{w}_{\epsilon}(x, t)=\mathbf{w}_0(x, t)+\epsilon\mathbf{w}_{1}(x, t)+O(\epsilon^{2})
$ into (\ref{perTransport}), first we find
 that $\mathbf{w}_{0}= 0$. Thus the limit problem for $\mb v_\e$ is given by
\begin{equation}
\p_t\bar{\mathbf{v}}(t)= \mathbb{B}\bar{\mathbf{v}}(t),\quad \bar{\mathbf{v}}(0) = \mathcal{P}\mathring{\mathbf{u}}\label{Picar2}\text{.}
\end{equation}
For $\mb w_1$ (which only is needed for a technical reason) we obtain
\begin{equation}
\partial_x \mathbf{w}_{1}= -\mathbb{B}\bar{\mathbf{v}}.\label{Kinetic4}
\end{equation}
Thus $\mathbf{w}_{1}(x, t) = \mathbb{B} \bar{\mathbf{v}}(t)x+ \mathcal{H}(t)$, where $\mathcal{H}(t)$ is arbitrary. Since $\mathbf{w}_{1}\in \mb W$, we find $\mathbf{w}_{1}(x, t) = \mathbb{B} \bar{\mathbf{v}}(t)\left({1}/{2}-x\right)$,
 where $\bar{\mathbf{v}}(t)=e^{\mathbb{B}\,t}\mathcal{P}\mathring{\mathbf{u}}$.

As in the diffusion case, the approximation $(\mb v_\e, \mb w_\e) \approx (\bar{\mb v}, \e\mb w_1)$ does not hold close to $t=0$ unless $\mathring{\mb u} = \mc P\mathring{\mb u} + O(\e)$. Thus, as usual, we introduce the initial layer  $\tilde{\mathbf{v}}(\tau) = \tilde{\mathbf{v}}_0(\tau)+ O(\epsilon)$ and $\tilde{\mathbf{w}}(x, \tau) = \tilde{\mathbf{w}}_0(x,\tau)+O(\epsilon)$, where $\tau = t/\epsilon$. Then standard argument gives $\tilde{\mathbf{v}}_0(x,\tau)=0$, while for the kinetic part we get
\begin{equation}
\partial_\tau \tilde{\mathbf{w}}_{0}(x,\tau) = -\partial_{x}\tilde{\mathbf{w}}_{0}(x,\tau),\,\,
\tilde{\mathbf{w}}_{0}(0, \tau) = \tilde{\mathbf{w}}_{0}(1, \tau),\,\, \tilde{\mathbf{w}}_{0}(x, 0) = \mathring{\mb w}= \mathring{\mathbf{u}}(x) - \mathcal{P}\mathring{\mathbf{u}}\text{.}\label{tildew0}
\end{equation}
We observe that, denoting by $(e^{\frac{t}{\e}{\mb A}_{\mbb I}})_{t\geq 0}$ the semigroup solving (\ref{tildew0}), we can write
  $$
 [e^{\frac{t}{\e}\mb A_{\mbb I}}\mathring{\mb w}](x) =  \mathring{\mb w}(n+x-\e^{-1}t),\qquad -n\leq x-\e^{-1}t\leq-n+1
 $$
 and, by estimate (\ref{estimp}),
 \begin{equation}
 \|e^{\frac{t}{\e}\mb A_{\mbb I}}\mathring{\mb w}\|_{\mb X} \leq \| \mathring{\mb w}\|_{\mb X}.
 \label{estper}
 \end{equation}

 \subsubsection{The error estimates}
 Let us define approximation as follows
 \begin{equation}
 (\mathbf{v}_{\epsilon}(t),\mathbf{w}_{\epsilon}(x,t)) = (\bar{\mathbf{v}}(t),\epsilon \mathbf{w}_1(x,t)+\tilde{\mathbf{w}}_0(x,\tau)) + (\mb e_{\mb v}(x,t), \mb e_{\mb w}(x,t)).
  \label{approxfin}
  \end{equation}
            As with the diffusion, we encounter the problem that, in general, a given initial condition $\mathring{\mb u}$, even if it belongs to $\mb W^1_1(I),$  it does not belong to $D(\mb A_\e)$ for all $\e>0$. Thus the situation is even worse than in the diffusion case, as the transport semigroup is not analytic  and thus we do not a differentiable solution for $t>0$, which is essential for the error estimates. Similarly, in general, the initial condition in (\ref{tildew0}) is not periodic so that the initial layer is not differentiable either. Thus, as before, we will work with approximate initial conditions.
  \begin{lemma}
  Let $\mb u\in \mb W^1_1(I)$. Then, for any $\delta\in ]0,1/2[$ there is $\mb u_\delta = \mc P\mb u_\delta + \mc Q\mb u_\delta = \mb v_\delta +\mb w_\delta \in \mb W^1_1(I)$ that satisfy
  \begin{eqnarray}
  \mb u_\delta(0) &=& \mb u_\delta(1) = 0,\label{c1}\\
  \mb w_\delta(0) &=& \mb w_\delta(1) =0,\quad \mb w_\delta \in \mb W, \label{c2}\\
  \|\mb u_\delta - \mb u\|_{\mb X} &\leq& C\delta\|\mb u\|_{\mb W^1_1(I)}, \label{c3}\\
  \|\mb u_\delta\|_{\mb X}&\leq& C\|\mb u\|_{\mb W^1_1(I)}, \label{c4}
  \end{eqnarray}
  for some constant $C$ independent of $\delta$. \label{lem42}
  \end{lemma}
\begin{proof}
Let $\mb u\in \mb W_{1}^{1}(I).$ Condition (\ref{c1}) can be achieved by a construction similar to that of Lemma \ref{udeltalem}.  One easily checks that  $\mb u_{\delta}$ defined for $0<\delta<1/2$ as
\be
\label{oszacowanie}
\hat{\mb u}_{\delta}(x)=\left\{ \begin{array}{ccc}
\mb u(x)& \text{for} &\quad x\in\left[\delta,1-\delta\right],\\
\frac{\mb u(\delta)}{\delta} x& \text{for}& \quad x\in\left[0,\delta\right],\\
\frac{\mb u(1-\delta)}{\delta}(1-x)& \text{for}& \quad x\in\left[1-\delta,1\right],
\end{array}\right.
\ee
    satisfies $\hat{\mathbf{u}}_\delta\in \mb W^{1}_{1}(I)$.  For arbitrary $x\in \left[0,1\right]$ we have \begin{equation}
    |\hat{\mb u}_{\delta}(x)|\leq C''\left\| \mb u \right\|_{\mb W^{1}_{1}(I)}\label{hatap}\end{equation} and
\bd
\left\|\mb u-\hat{\mb u}_{\delta}\right\|_{\mb X} \leq \frac{\delta}{2}\left[ \mb u(\delta)+\mb u(1-\delta)\right]\leq C'\delta \left\| \mb u \right\|_{\mb W_{1}^{1}(I)}.
\ed
However, for $[\mc Q\hat{\mb u}_\delta]$ we only have
$[\mc Q\hat{\mb u}_\delta](0)= [\mc Q\hat{\mb u}_\delta](1);$
that is, the projection of $\hat{\mb u}_\delta$ onto $\mb W$ is periodic but does not satisfy (\ref{c2}). Thus, we define
\begin{equation}
\mb u_\delta(x) = \mb v_\delta(x) + \mb w_\delta(x) = \cl{0}{1}\hat{\mb u}_\delta(x)dx  + \left(\hat{\mb u}_\delta(x) - \bom_\delta (x)\int_0^1\hat{\mb u}_\delta(x)dx\right),
\label{finap}
\end{equation}
where
\be
\label{oszacowanie1}
\hat{\bom}_{\delta}(x)=\left\{ \begin{array}{lcl}
\frac{1}{1-\delta}& \text{for}& \quad x\in\left[\delta,1-\delta\right],\\
\frac{x}{(1-\delta)\delta} & \text{for}& \quad x\in\left[0,\delta\right],\\
\frac{1-x}{(1-\delta)\delta}& \text{for}&
 \quad x\in\left[1-\delta,1\right].
\end{array}\right.
\ee
Since $\int_0^1\bom_\delta(x)dx = 1$, we have $\mb w_\delta \in \mb W$. Then we have
\begin{eqnarray*}
\|\mb u -\mb u_\delta\|_{\mb X}\!\! &\leq & \!\!\|\mc P\mb u -\mc P\hat{\mb u}_\delta\|_{\mb X} + \|\mb u-\hat{\mb u}_\delta\|_{\mb X} + \left\| \cl{0}{1} \mb u(x)dx - \bom_\delta\cl{0}{1} \hat{\mb u}_\delta(x)dx\right\|_{\mb X}\\
&\leq& 2C'\delta\|\mb u\|_{\mb W^1_1(I)} + \cl{0}{1}\|\mb u(x)-\hat{\mb u}_\delta(x)\|_{\mb X}dx \\
&& \phantom{x}+ \|1-\bom_\delta\|_{\mb X}\cl{0}{1}|\hat{\mb u}_\delta(x)|dx\\
&\leq& 3C'\delta\|\mb u\|_{\mb W^1_1(I)} + 2C''\delta(1-\delta)\|\mb u\|_{\mb W^1_1(I)} \leq C\delta\|\mb u\|_{\mb W^1_1(I)}.
\end{eqnarray*}
Inequality (\ref{c4}) follows by integrating (\ref{finap}) and using (\ref{hatap}).
\end{proof}
Now we can formulate the main theorem.
\begin{thm}
\label{glowne}
  For any $T\in\left]0,\infty\right[$ there exists $C(T, \mbb B)$ such that for any (sufficiently small) $\e>0$ and  $\mathring{\mathbf{u}}\in \mb W_{1}^{1}(I)$ the solution $\mb u_\e(t) = e^{t\mb A_\e}\mathring{\mb u}$ of (\ref{Transport}) satisfies
  \begin{equation}
  \|\mb u_\e(t) - \bar{\mb v}(t) - \ti{\mb w}_0(t/\e)\|_{\mb X}\leq \e C(T,\mbb B)\|\mathring{\mb u}\|_{\mb W^1_1(I)},
  \label{oszacowanie``}
  \end{equation}
  uniformly on $[0,T],$ where $\bar{\mathbf{v}}$  and $\tilde{\mathbf{w}}_{0}$ solve, respectively, (\ref{Picar2}) and (\ref{tildew0}).
\end{thm}
\begin{proof}
Let us now consider (\ref{Transport''}) with $\mathring{\mb u} \in \mb W^1_1(I)$ replaced by $\mathring{\mb u}_\delta=(\mathring{\mb v}_\delta, \mathring{\mb w}_\delta)$ constructed as in Lemma \ref{lem42} and consider the corresponding approximation of $\mb u_{\e \delta}= (\mb v_{\e\delta}, \mb w_{\e\delta})$ given by $(\bar{\mathbf{v}}_\delta(t),\epsilon \mathbf{w}_{1\delta}(x,t)+\tilde{\mathbf{w}}_{0\delta}(x,\tau))$. As before, the error is defined by
${\mathbf{e}}_{\mathbf{v}\delta} (t)= \mathbf{v}_{\epsilon\delta}(t)-\bar{\mathbf{v}}_\delta(t)$ and ${\mathbf{e}}_{\mathbf{w}\delta}(x,t) = \mathbf{w}_{\epsilon\delta}(x,t)-\epsilon \mathbf{w}_{1\delta}(x,t)-\tilde{\mathbf{w}}_{0\delta}(x,\tau)\text{.}
$
Thanks to the construction of $\mathring{\mb u}_\delta$, all terms in the error equation are differentiable and satisfy the boundary conditions for  $t\in [0,\infty[$, in particular, we have $\ti{\mb w}_{0\delta}(0,\tau)=0$. Hence, by direct substitution, using (\ref{Kinetic4}) and (\ref{tildew0}), we find that the error is a classical solution of the following problem
\begin{eqnarray*}
\p_t{\mathbf{e}}_{\mathbf{v}\delta}(t) &=& \mathbb{B}{\mathbf{e}}_{\mathbf{v}\delta}(t) + \mathbb{B}{\mathbf{e}}_{\mathbf{w}\delta}(0,t)+\e\mathbb{B} \mathbf{w}_{1\delta}(0, t),\\
\partial_{t}{{\mathbf{e}}_{\mathbf{w}\delta}}(x,t) &=& -\frac{1}{\epsilon}\partial_x{\mathbf{e}}_{\mathbf{w}\delta}(x,t) - \mathbb{B}{\mathbf{e}}_{\mathbf{v}\delta}(t) - \mathbb{B}{\mathbf{e}}_{\mathbf{w}\delta}(0,t)-\epsilon \mbb B\mathbf{w}_{1\delta}(0,t)\\&&-\epsilon \partial_t\mathbf{w}_{1\delta}(x,t),\\
{\mathbf{e}}_{\mathbf{w}\delta}(0, t)& =& \mb{e}_{\mathbf{w}\delta}(1,t)+\epsilon\mathbb{B}\mb{e}_{\mathbf{v}\delta}(t)+\epsilon \mathbb{B}\mb{e}_{\mathbf{w}\delta}(1,t)+\frac{\epsilon^{2}}{2}\mathbb{B}^{2}\bar{\mathbf{v}}_\delta(t),
\\
\mb e_{\mb v\delta}(x,0)&=&0,\quad
{\mathbf{e}}_{\mathbf{w}\delta}(x, 0) = -\epsilon \mathbf{w}_{1}(x,0).
\end{eqnarray*}
 Combining these two equations for the single error
  ${\mathbf{E}_\delta} = {\mathbf{e}}_{\mathbf{v}\delta}+{\mathbf{e}}_{\mathbf{w}\delta}$, we get
\begin{eqnarray}
\partial_{t}{\mathbf{E}_\delta} &=& -\frac{1}{\epsilon}\partial_x{\mathbf{E}_\delta} - \epsilon \partial_t\mathbf{w}_{1\delta},\label{errorfin}\\
{\mathbf{E}_\delta}(0, t) &=&{\mathbf{E}}_\delta(1,t)+\epsilon \mathbb{B} {\mathbf{E}_\delta}(1,t)+\frac{\epsilon^{2}}{2}\mathbb{B}^{2}\bar{\mathbf{v}}_\delta(t),
\quad
{\mathbf{E}_\delta}(x, 0) =-\epsilon \mathbf{w}_{1\delta}(x,0).\nn
\end{eqnarray}
To be able to use the semigroup estimates for (\ref{errorfin}), we have to lift the inhomogeneity in the boundary condition to the equation. To do this, we define $
\mb F_\delta (x,t)=\mb E_\delta(x,t) - {\epsilon^{2}}(1-x)\mathbb{B}^{2}\bar{\mathbf{v}}_\delta(t)/2.
$
Then $\mb F_\delta$ satisfies
\begin{eqnarray*}
\partial_{t}{\mathbf{F}_\delta}(x,t) &=& -\frac{1}{\epsilon}\partial_x{\mathbf{F}_\delta}(x,t)+\e \mb J(x,t),\nn\\
{\mathbf{F}_\delta}(0, t)&=&{\mathbf{F}_\delta}(1,t)+\epsilon \mathbb{B} {\mathbf{F}_\delta}(1,t),\quad
{\mathbf{F}_\delta}(x, 0)=\e \mb H(x),
\end{eqnarray*}
where
\begin{eqnarray*}
\mb J(x,t)&=&\frac{1}{2}\mathbb{B}\bar{\mathbf{v}}_\delta(t)-\frac{\epsilon}{2}(1-x)\mathbb{B}^3
\bar{\mathbf{v}}_\delta(t) -\left(\frac{1}{2}-x\right)\mbb B^2\bar{\mb v}_\delta(t),\\
\mb H(x) &=& -\left(\frac{1}{2}-x\right)\mbb B\mathring{\mathbf{v}}_\delta(x)-\frac{\e(1-x)}{2}\mbb B\mathring{\mb v}_\delta(x).
\end{eqnarray*}
Thus, using the Duhamel formula and (\ref{unibd}), we obtain
\begin{eqnarray*}
\|\mb E_\delta(t)\|_{\mb X} &\leq& \|\mb F_\delta(t)\|_{\mb X} + \frac{\e^2}{4}\|\mbb B\|^2e^{\|\mbb B\|T}\|\mathring{\mb v}_\delta\|_{\mb X}\leq \frac{3\e(1+\e)M\|\mbb B\|}{4}\|\mathring{\mb v}_\delta\|_{\mb X}\\
&&+\frac{\e M(e^{\|\mbb B\|T}-1)}{4}(2 +\e\|\mbb B\|^2+ 2\|\mbb B\|)\|\mathring{\mb v}_\delta\|_{\mb X}\\
 &&+ \frac{\e^2}{4}\|\mbb B\|^2e^{\|\mbb B\|T}\|\mathring{\mb v}_\delta\|_{\mb X}\leq  C'_T\e \|\mathring{\mb v}_\delta\|_{\mb X}
\end{eqnarray*}
for some constant $C'_T$ independent of $\e$ and the initial condition.

To complete the proof, we take $\mathring{\mb u}\in \mb W^1_1(I)$ and, for any given $\e>0$ we construct the approximation $\mathring{\mb u}_\delta$ with  $\delta < \min\{\e, 1/2\}$. Further, let $\bar{\mb v}+\e\mb w_1 + \ti{\mb w}_0$ be the approximation constructed by (\ref{Picar2}), (\ref{Kinetic4}) and (\ref{tildew0}) with $\mathring{\mb v}=\mc P \mathring{\mb u}, \mathring{\mb w}=\mc Q \mathring{\mb u}$.  Then, denoting by $\mb E$ the error of this approximation, we find
\begin{eqnarray*}
\|\mb E(t)\|_{\mb X} &\leq& \|e^{t\mb A_\e}\mathring{\mb u} - e^{t\mbb B}\mathring{\mb v} -\e\mb w_{1}(t) - \ti{\mb w}_{0}(\tau)\|_{\mb X}\\
&\leq&  \|e^{t\mb A_\e}\mathring{\mb u}_\delta - e^{t\mbb B}\mathring{\mb v}_\delta -\e\mb w_{1\delta}(t) - \ti{\mb w}_{0\delta}(\tau)\|_{\mb X} +
\|e^{t\mb A_\e}(\mathring{\mb u}-\mathring{\mb u}_\delta)\|_{\mb X} \\
&&+ \|e^{t\mbb B}(\mathring{\mb v} -\mathring{\mb v}_\delta)\|_{\mb X}
 +\e\|\mb w_{1}(t)-\mb w_{1\delta}(t)\|_{\mb X} +\| \ti{\mb w}_{0}(\tau)-\ti{\mb w}_{0\delta}(\tau)\|_{\mb X}\\
 &\leq& C'_T\e \|\mathring{\mb v}_\delta\|_{\mb X} + Me^{\|\mbb B\|T}\|\mathring{\mb v} -\mathring{\mb v}_\delta\|_{\mb X} +\frac{\e e^{\|\mbb B\|T}}{2} \|\mathring{\mb v} -\mathring{\mb v}_\delta\|_{\mb X} \\&&+ \|\mathring{\mb w}-\mathring{\mb w}_\delta\|_{\mb X}\leq \e C''_T\|\mathring{\mb u}\|_{\mb W^1_1(I)},
\end{eqnarray*}
where in the last line we used (\ref{c3}), (\ref{c4}) and $\delta<\e$. Then, noting that $\e\|\mb w_1(t)\|_{\mb X} \leq \e {e^{T\|\mbb B\|}}/{2}$, we find that (\ref{oszacowanie``}) is true.\end{proof}

\begin{remark}
We observe that this result is of a different type than Theorem \ref{mainth1}. In the latter, the initial layer term $\ti{\mb w}_0$ decays exponentially to 0 as $\e \to 0$ for any $t>0$, see (\ref{F-solution}). Indeed, let $t\in [t_0,T], t_0>0$, $0<\e<\e_0$. Then (\ref{F-solution}) can be estimated as
$$
\ti{\mb w}_0(t/\e) \leq C_1(t_0,\e_0, \mathring{\mb u}) e^{-\frac{\pi^2 t}{\e}} \leq \e\frac{ C_1(t_0,\e_0, \mathring{\mb u})}{\pi^2 t_0}\frac{\pi^2 t}{\e} e^{-\frac{\pi^2 t}{\e}}\leq C_2\e
$$
where $C_2$ is independent of $t$ and $\e$. Hence (\ref{finest}) can be written as
$$
 \|\mb u_\e(t) - \bar{\mb v}(t)\|_{\mb X}\leq \e C_3
$$
uniformly on $[t_0,T]$, where $C_3$ is independent of $\e$ and $t,$ but depends on $\mathring{\mb u}$, $t_0$ and the coefficients of the boundary conditions. In other words, outside an $O(\e)$ transition zone, the whole solution to the PDE problem on a network (\ref{system1'}) can be approximated by the solution of an ODE system (\ref{barv}).

This is in contrast to (\ref{oszacowanie}), where the term $\ti{\mb w}_0$ does not decay exponentially with $\e \to 0$ but is periodic with period $\e$ (in $t$). Hence, in the transport problem (\ref{Transport''}), the solution $\mb u_\e(t)$ cannot be approximated by the solution $\bar{\mb v}$ to the Cauchy problem for an ODE, defined by (\ref{Picar2}). However, using the fact that $\mb X = \mb V \oplus \mb W$, we can project (\ref{oszacowanie}) to find that
 \begin{equation}
  \|\mb v_\e(t) - \bar{\mb v}(t)\|_{\mb X}\leq \e C(T,\mbb B)\|\mathring{\mb u}\|_{\mb W^1_1(I)},\,\,
     \|\mb w_\e(t) - \ti{\mb w}_0(t/\e)\|_{\mb X}\leq \e C(T,\mbb B)\|\mathring{\mb u}\|_{\mb W^1_1(I)}.
    \label{oszacowanie'}
  \end{equation}
  Thus, the macroscopic characteristics of the flow on the network; that is, the mass on each edge, can be approximated by the solution of  ODE (\ref{Picar2}). The reminder, however, approximately behaves as a fast oscillating function with zero mean.
\end{remark}
\begin{remark}
In the simple case of unit speeds along each edge, the convergence ensured by the first estimate of (\ref{oszacowanie'}) can be proved directly. Recall that
 \bd
[e^{\mb A_{\epsilon} t}\mathring{\mathbf{u}}](x)=(\mathbb{I}+\epsilon \mathbb{B})^{n}\mathring{\mathbf{u}}\left(n+x-\frac{1}{\epsilon}t\right) \qquad \text{for} \quad -n\leq x-\frac{1}{\epsilon}t\leq-n+1.
\ed
Hence, for $n-1 \leq t/\e\leq n$ we have, as in (\ref{estimp}),
$$
\mc P [e^{\mb A_{\epsilon} t}\mathring{\mathbf{u}}]
%(\mathbb{I}+\epsilon \mathbb{B})^{n}\!\!\!\!\!\!\!\cl{0}{\frac{t}{\epsilon}-n+1}\!\!\!\!\!\mathring{\mb u}\left(n+x-\frac{t}{\epsilon}\right)dx+
%(\mathbb{I}+\epsilon \mathbb{B})^{n-1}\!\!\!\!\!\!\!\cl{\frac{t}{\epsilon}-n+1}{1}\!\!\!\!\!\mathring{\mb u}\left(n-1+x-\frac{t}{\epsilon}\right)dx %\nn\\
=(\mathbb{I}+\epsilon \mathbb{B})^{n-1}\!\!\cl{0}{1}\!\!\mathring{\mb u}\left(z\right)dz+\e\mbb B(\mathbb{I}+\epsilon \mathbb{B})^{n-1}\!\!\!\cl{n-\frac{t}{\epsilon}}{1}\!\!\!\mathring{\mb u}\left(z\right)dz.
$$
We have  $t = (n-1)\e +t',$ where $t'\in [0,\e]$; that is, $n-1 = t/\e +\theta(\e)$ with $0\leq \theta(\e) \leq 1$.
Since $\mbb B$ is a matrix, using the Dunford functional calculus we obtain
$$
\lim\limits_{\e\to 0} (\mathbb{I}+\epsilon \mathbb{B})^{n-1} = \lim\limits_{\e\to 0} (\mathbb{I}+\epsilon \mathbb{B})^{\frac{t}{\e}}(\mbb I+\e\mbb B)^{\theta(\e)} = e^{t\mbb B},\quad \lim\limits_{\e\to 0} \e\mbb B(\mathbb{I}+\epsilon \mathbb{B})^{n-1} =0.
$$
Hence
$$
\lim\limits_{\e\to 0}\mc P [e^{\mb A_{\epsilon} t}\mathring{\mathbf{u}}] = e^{t\mbb B}\mc P\mathring{\mb u},
$$
where $e^{t\mbb B}\mc P\mathring{\mb u}$ is the solution to (\ref{Picar2}). We observe that this result was obtained without assumption that $\mathring{\mb u}\in \mb W^1_1(I)$ but also it does not yield the rate of convergence ensured in (\ref{oszacowanie'}). In fact, using the density of $\mb W^1_1(I)$ in $\mb X$ and uniform boundedness with respect to $\e$ of all involved operators, this result can be deduced from (\ref{oszacowanie'}) by the $3-\e$ lemma (a corollary to the Banach-Steinhaus theorem), as  in (\ref{weak}).
\end{remark}

\subsubsection{Some pitfalls of constructing micro-models.}
\label{pitfal}
In the first two examples we consider the transport model (\ref{ACP1b})  on a strongly connected digraph $G$ so that the matrix $\mathbb{T}$ is irreducible and column stochastic \cite{BaP,KS04}.
\begin{example}
 First we try to mimic the scaling of (\ref{system1'}) and consider
\begin{equation}
\partial_{t}\mathbf{u}_\e+\frac{1}{\epsilon}\partial_{x}\mathbf{u}_\e = 0,\quad \mathbf{u}_\e(0, t)= \epsilon \mathbb{T}\mathbf{u}_\e(1, t), \quad
\mathbf{u}_\e(x, 0)  = \mathring{\mb u}(x).\label{Transport}
\end{equation}
It is clear that the hydrodynamic space of this problem, spanned by the solutions $
\p_x\mb u =0, \, \mb u(0,t)=0,
$
only consists of the zero function and thus (\ref{Transport}) does not offer any interesting limit dynamics.
\end{example}
       \begin{example}     \label{ex52}   In the next case  the boundary condition is changed as follows,
\begin{equation}
\partial_{t}\mathbf{u}_\e+\frac{1}{\epsilon}\partial_{x}\mathbf{u}_\e = 0,\quad \mathbf{u}_\e(0, t)= \mathbb{T}\mathbf{u}_\e(1, t), \quad
\mathbf{u}_\e(x, 0)  = \mathring{\mb u}(x).\label{Tran'}
\end{equation}
    In this case, the hydrodynamic space $\mb V$ consists of solutions to
    $
    \p_x\mathbf{u} = 0,
     \mb u(0,t) = \mbb T\mb u(1,t).
    $
Proceeding as in previous section, we find that the projection of $\mathbf{u}$ onto $\mb V$ is given by
$\mathcal{P}\mathbf{u} = \left(\mb 1\cdot \int_{0}^{1}\mb u(x)dx\right)\mathbf{N},
$
 where $\mathbf{N}$ is the Perron eigenvector of the matrix $\mathbb{T}$ and $\mb 1 = (1,\ldots,1)$ is its left eigenvector; we assumed that $\mb N$ is normalised so that  $\mb 1\cdot\mb N =1$. Then $\mathbf{v}_{\epsilon} = \mathcal{P}\mathbf{u}_{\epsilon}$ satisfies
$$
 \partial_t \mb v_\e= \frac{1}{\epsilon}\mathcal{P}\p_x\mathbf{u}_\e  = \frac{1}{\epsilon}\left((\mathbf{u}(1, t)-\mathbf{u}(0, t))\cdot\mb 1\right)\mb N = \frac{1}{\epsilon} \mathbf{u}(1, t)\cdot(\mbb I -\mathbb{T}^{T})\mathbf{1} = 0.
$$
Hence, the equation for the projection of the solution onto $\mb V$ (the hydrodynamic part of the solution) is exactly the limit equation and $\mb v_\e = \rho \mb N,$ where $\rho = \left(\sum_{i=1}^{m}\int_{0}^{1}\mathring{u}_i(x)dx\right)$ is the initial mass. Then it is easy to see that the complementary projection (the kinetic part) $\mb w$ coincides with the initial layer $\ti {\mb w}$. However, the initial layer equation is identical to (\ref{Tran'}) (with $\e=1$) but with the initial condition in the kernel of $\mc P$. The asymptotic theory of such problems has been developed in e.g. Refs.  \cite{BaP,KS04,MS07} and, in general, it follows that $\ti{\mb w}$ does not have an exponential decay as $\tau \to 0$ ($\e\to 0$). We obtain the (exact) decomposition
$$
\mb u_\e (x,t) = \rho \mb N + \ti{\mb w}(x,t/\e)
$$
 which somehow resembles  Theorem \ref{glowne}. However, the hydrodynamic part of the solution gives the stationary distribution of the total mass on the network among the edges. Hence, it provides neither simplification of the original problem nor any interesting approximate dynamics.

\end{example}

\begin{example}\label{ex13}\textbf{Structured McKendrick model \cite{BaLabook,BSG,BSG2,BG}.}  Consider a population divided into  $m$ groups with respect to some attribute $i$  and assume that the individuals can move between them. These groups could refer  e.g. to geographical patches (and then the change of the group would be due physical migration), or to the number of  particular genes (whereupon the change can occur due to mutation).
 If the vector $\mb n(t) = (n_1(t),\ldots,n_m(t))  \in \mbb R^m$  gives the numbers of individuals  in groups $1,\ldots, m$  at time $t$,
then
\begin{equation}
\p_t\mb n= \mbb K\mb n, \qquad \mb n(0) = \mathring{\mb n},
\label{fastdyn1}
\end{equation}
where $\mbb{K} =\{k_{ij}\}_{1\leq i,j\leq m}$ is a
Kolmogorov transition matrix (of a time-continuous process); that is, nonnegative off-diagonal and the columns sum up to 0.
Assume now that we also want to include the demographical processes in  each patch and that  migrations occur at a much faster rate than the demographic processes. Let us  denote by $n_{i,\e}$ the
population density in patch $i$ and by $a$ the age of individuals. Then in each patch the density $n_{i,\e}(a,t)$ should satisfy the McKendrick equation and hence  we arrive at the system
\begin{equation}
\p_t\mb{n}_\e =\mb{A}\mb{n}_\e +\mbb {M}\mb{n}_\e + \frac{1}{\e}\mbb{K}\mb{n}_\e,\quad
\mb n_\e(0,t)= \cl{0}{\infty}\mbb B(a)\mb n_\e(a,t)da,\quad \mb n(a,0)= \mathring{\mb n},
\label{pertsola1'}
\end{equation}
where $\mb A = \mathrm{diag}\left\{-\p_a, \ldots,-\p_a\right\}$, $\mbb{B}(a)=\mathrm{diag}\{\beta_{j}(a)\}_{1\leq j\leq m}$ describes
        the age and patch specific fertility rates,  $\mbb M(a) = \mathrm{diag}\{-\mu_{j}(a)\}_{1\leq i,j\leq m}$,
where $\mu_j(a)$ is the age specific death rate in patch $j$ and $1/\e$ is the ratio of the reference times of demographic and migration processes.

Note that here we have a simplified version of the model from Fig. \ref{fig3}, where the exchange between the nodes occurs instantaneously according to the migration rates $k_{ij}$, $1\leq i,j \leq m, i\neq j$.

Biological heuristics suggests that no group structure
should persist for very large intergroup transition rates; that
is, for $\e\to 0$. The precise result depends on the structure of the network of connections described by the matrix $\mbb K$. In particular, if the corresponding network is strongly connected; that is,  if $\mbb K$ is irreducible,  then $\la = 0$ is the  dominant
simple eigenvalue of $\mbb{K}$ with  $\mb{1} =(1,1,\ldots,1)$ as a left eigenvector and a
corresponding positive right (Perron) eigenvector, denoted by $\mb{N}$, normalized so as to $\mb{1}\cdot\mb{N}=1$. The vector
$\mb{N} =(N_1,\cdots,N_m)$ is so-called the stable patch
distribution; that is, the asymptotic, as $t\to \infty$ and
disregarding demographic processes, distribution of the population described by (\ref{fastdyn1})
among the groups. Under some technical assumptions \cite{BSG,BG,BSG2}, it can be proved that the total population $n_\e = n_{1,\e}+\ldots+n_{m,\e}$ converges as $\e\to 0$ to the solution to
 \begin{equation}
 {\p_t{n}}= -{\p_a{n}} - \mu^*n,\quad
 n(0,t)=\cl{0}{\infty}\beta^*(a)n(a,t)da,\quad  n(a,0)= \mb 1\cdot\mathring {\mb n},
 \label{sola2}
\end{equation}
where  $\mu^* = \mu_1 N_1+\ldots +\mu_m N_m$ and $\beta^* = \beta_1 N_1+\ldots +\beta_m N_m$ are, respectively, the  `aggregated'
mortality and fertility rates. We observe that, as required by the general paradigm, (\ref{sola2}) does not display any explicit `space' structure. However, a  shadow of the original network persists in the model through the coefficients of the stable patch distribution vector $\mb N.$
More precisely, the coefficients of $\mb N$ give approximate fractions of the population residing on average in each patch and thus subject to patch specific mortality and birth processes. Therefore $\mu^*$ and $\beta^*$ are aggregated death and birth coefficients which take into account that different fractions of the population die and give birth with different rates.

At the same time we observe that (\ref{pertsola1'}), despite being constructed as a microscopic extension of (\ref{fastdyn1}), does not yield it back in the singular limit.
\end{example}

\end{document}